\documentclass{article}
\usepackage[english]{babel}

\usepackage{graphicx}
\usepackage[pdfborder={0 0 0}]{hyperref}
\usepackage[utf8]{inputenc}

\usepackage{geometry}
\usepackage{mathtools}
\usepackage{amssymb}
\usepackage{dsfont}
\usepackage{amsmath}
\usepackage{url}
\usepackage{float}
\usepackage{comment}
\usepackage[T1]{fontenc}
\usepackage{bm}
\usepackage{graphicx}
\usepackage[normalem]{ulem}
\usepackage{commath}
\usepackage{amsthm}
\usepackage{mathrsfs}  
\usepackage{amsmath,amscd}
\usepackage{hyperref}
\usepackage{bm}
\usepackage{array}
\usepackage[nottoc,notlot,notlof]{tocbibind}
\usepackage[title]{appendix}
\usepackage{tikz}
\usepackage[title]{appendix}
\usepackage{xcolor}
\usepackage{leftidx}
\usepackage{calc}

\usepackage{accents}

\usepackage{pgfplots}

\usepackage{cleveref}

\usepackage{tikz-cd}
\usetikzlibrary{matrix,arrows}
\usepackage{subfiles}

\newcommand{\doubletilde}[1]{\widetilde{\raisebox{0pt}[0.85\height]{$\widetilde{#1}$}}}

\newcommand{\pa}[1]{\left( #1 \right)}
\newcommand{\op}[1]{\left\| #1 \right\|}

\newcommand{\N}{\mathbb{N}}
\newcommand{\Z}{\mathbb{Z}}
\newcommand{\Q}{\mathbb{Q}}
\newcommand{\C}{\mathbb{C}}
\newcommand{\R}{\mathbb{R}}

\newcommand{\Hil}{\mathbb{H}}

\newcommand{\Span}{\mathrm{Span}}

\newcommand{\re}{\mathrm{Re}}
\newcommand{\Diag}{\mathrm{diag}}

\newcommand{\M}{\mathcal{M}}

\newcommand{\ev}{\mathrm{ev}}

\newcommand{\supp}{\mathrm{supp}}

\newcommand{\uperp}{\widetilde{u}^{k,\ell}_\perp}

\newcommand{\cloc}{C^{m-1}_{\mathrm{loc}}}

\makeatletter
\newcommand{\opnorm}{\@ifstar\@opnorms\@opnorm}
\newcommand{\@opnorms}[1]{%
  \left|\mkern-1.5mu\left|\mkern-1.5mu\left|
   #1
  \right|\mkern-1.5mu\right|\mkern-1.5mu\right|
}
\newcommand{\@opnorm}[2][]{%
  \mathopen{#1|\mkern-1.5mu#1|\mkern-1.5mu#1|}
  #2
  \mathclose{#1|\mkern-1.5mu#1|\mkern-1.5mu#1|}
}
\makeatother

\pgfplotsset{compat=1.10}
\usepgfplotslibrary{fillbetween}
\usetikzlibrary{patterns}

\newcommand{\wt}[1]{\widetilde{#1}}

\def\[#1\]{\begin{align*}#1 \end{align*}}

\newcommand{\kolomtwee}[2]{\left ( \begin{matrix} #1\cr #2 \end{matrix} \right)}

\DeclarePairedDelimiter\floor{\lfloor}{\rfloor}

\newtheorem{thm}{Theorem}[section]
\newtheorem{definition}[thm]{Definition}
\newtheorem{theorem}[thm]{Theorem}
\newtheorem{proposition}[thm]{Proposition}
\newtheorem{corollary}[thm]{Corollary}
\newtheorem{lemma}[thm]{Lemma}

\newtheorem{example}[thm]{Example}
\newtheorem{maintheorem}[thm]{Main Theorem}

\newtheorem*{remark}{Remark}

\title{Time-periodic solutions of Hamiltonian PDEs using pseudoholomorphic curves}
\author{Oliver Fabert \\ Niek Lamoree}
\date{ }

\begin{document}

\maketitle
\vspace{-8mm}
\abstract{We extend the pseudoholomorphic curve methods from Floer theory to infinite-dimensional phase spaces and use our results to prove the existence of a forced time-periodic solution to a general Hamiltonian PDE with regularizing nonlinearity. In particular, when the nonlinearity is sufficiently regularizing, bounded and time-periodic, we prove an infinite-dimensional version of Gromov-Floer compactness by using ideas from the theory of Diophantine approximations to overcome the small divisor problem. %We then use our result to prove that for a generic time period a solution exists, thereby generalizing the result of Floer-Hofer from \cite{fh1} to infinite dimensions and complementing the result by Rabinowitz \cite{rabinowitz}. 
Furthermore, in the case when the infinite-dimensional phase space is a product of a finite-dimensional closed symplectic manifold with linear symplectic Hilbert space, we prove a cup-length estimate for the number of periodic solutions.}

\tableofcontents

\section*{Introduction}
It is well-known that problems in classical mechanics can be formalized and solved in the language of Hamiltonian systems on finite-dimensional symplectic manifolds, or phase spaces. This led to a rather novel mathematical branch called symplectic topology. Most of the ground-breaking results in symplectic topology rely on the existence of so-called pseudoholomorphic curves, that is, maps from a Riemann surface into the finite-dimensional symplectic manifold equipped with a compatible almost complex structure, satisfying a Cauchy-Riemann-type equation (see \cite{mcduffjholo}). They were introduced by M. Gromov in his seminal paper \cite{gromov1985}. In order to prove the existence of time-periodic solutions of Hamiltonian systems on finite-dimensional phase spaces, A. Floer has developed the tool of Floer homology which is based on the study of moduli spaces of so-called Floer curves which satisfy a Cauchy-Riemann-type equation involving a zeroth order Hamiltonian term (see \cite{floer1989}). The key technical result on which his theory relies is a compactness result for the moduli space of Floer curves, called Gromov-Floer compactness which, among other things, crucially uses compactness of the target manifold (see also \cite{fh1}). %It crucially relies on the fact that all Floer curves have image in a compact set such as a closed (finite-dimensional) symplectic manifold. In order to generalize Floer's methods from the case of closed symplectic manifolds to (finite-dimensional) linear symplectic spaces, one employs a maximum principle to compensate for the lack of compactness. This leads to a variant of Floer homology called symplectic homology (see \cite{fh1}).     

Generalizing from point particles to continuous fields, and hence from classical mechanics to classical field theory, one arrives at Hamiltonian systems which are defined on infinite-dimensional phase spaces, such as symplectic Hilbert spaces. More generally, it is known that a number of important partial differential equations, such as the (nonlinear) Schrödinger equation, the (nonlinear) wave equation, the Korteweg-de Vries equation, and many more, can be viewed as infinite-dimensional Hamiltonian systems. Such partial differential equations are also called \emph{Hamiltonian partial differential equations} (see \cite{kuksinbook}). We stress that it is a general feature of Hamiltonian PDEs that the Hamiltonian is typically only densely defined on the symplectic Hilbert space. In these cases, the best thing we can expect is an sc-Hamiltonian system in the sense of \cite{hwzsc}. 

In this paper we show how the pseudoholomorphic curve methods of Hamiltonian Floer theory can be generalized to the infinite-dimensional setup of Hamiltonian PDE with so-called regularizing nonlinearities. As a first step towards generalizing Floer theory from finite dimensions to infinite dimensions, we prove a version of Gromov-Floer compactness and also readily use our result to establish the existence of time-periodic solutions. There has been a lot of great work on finding solutions, see e.g. \cite{rabinowitz}, or \cite{berti} for an overview of the current state of the art (also see \Cref{subqsection} for more references). In this paper our aim is to show how pseudoholomorphic methods can be applied to this problem. An obvious problem comes from the fact that the Hamiltonian is only densely defined on the symplectic Hilbert space. As it turns out, one of the main new arising challenges is a small divisor problem: while for generic time and space periods the underlying linear Hamiltonian PDE only has the trivial periodic solution and the return map only has eigenvalues different from one, there is always a subsequence of eigenvalues converging to one. 

Apart from showing that the bubbling-off argument still works in order to uniformly bound derivatives, as main result we show how regularizing the nonlinearity of a Hamiltonian PDE needs to be in order to guarantee that Gromov-Floer compactness still holds. It turns out that this is intimately related with the aforementioned small divisor problem and ultimately with the theory of diophantine approximations. We define the concept of (weakly) admissible nonlinearities in order to classify the types of nonlinearities for which Gromov-Floer compactness can still be established and we further study the regularity of the Floer curves and the time-periodic solution in both the flow and time coordinates as well as the extra spatial coordinate.

We want to emphasize that this paper is mainly addressed at researchers with a background in Hamiltonian Floer theory, who are interested in the generalization of these techniques to the infininte-dimensional case of Hamiltonian PDEs. While we cite some well-established results from finite-dimensional Floer theory without proof, this paper is written in such a way that we do not assume any prior knowledge about Hamiltonian partial differential equations, small divisor problems and Diophantine approximations. In particular, we make no claim that the results about periodic solutions could not be obtained using different methods. Our ultimate goal is to construct a full Floer homology theory for Hamiltonian PDEs with regularizing nonlinearities. As a first result which needs pseudoholomorphic curve methods, we prove a cup-length estimate for a Hamiltonian system on a phase space which is a product of a closed finite-dimensional symplectic manifold with linear symplectic Hilbert space.\\\par

This paper is organized as follows: in \Cref{introduction} we give a brief introduction to nonlinear Hamiltonian PDEs and establish notation. There and in \Cref{hilbertdiophantine} we give a number of examples. In \Cref{hilbertdiophantine} we furthermore discuss the part of the Hamiltonian which contains the differential operator, as well as illustrate the so-called small divisor problem which occurs when passing to infinite dimensions. In the same section and in \Cref{admissiblenonlin}, we give a number of admissibility conditions and define the class of Hamiltonians for which our results hold, and we give a counterexample to show that some of these conditions cannot be relaxed. \par
\Cref{mainthmsection} contains a summary of the main results. In \Cref{fdsection} we recall well-established results from finite-dimensional Floer theory in order to establish the existence of Floer curves in finite-dimensional linear symplectic spaces. In \Cref{bubblingoffsection} we generalize the bubbling-off analysis for finite-dimensional pseudoholomorphic curves and show that the derivatives of the sequence of Floer curves are bounded; this includes a standard elliptic regularity argument to include higher derivatives. Using a series of estimates, \Cref{sdp} shows how the higher-dimensional components of Floer curves can be controlled in the presence of the small divisor problem. In \Cref{completingproof} we complete the proof of the main theorem and subsequently generalize this to a wider class of Hamiltonians in \Cref{subqsection}. \Cref{cuplengthsection} provides a cup-length estimate for the number of periodic solutions when we consider Hamiltonian systems on the product of a finite-dimensional closed symplectic manifold with linear symplectic Hilbert space. Finally there are two appendices, the first of which introduces sc-Hamiltonian flows, and the second containing details on Hamiltonian curvature which is used to establish our compactness theorem.

\section{Nonlinear Hamiltonian PDEs}
\label{introduction}
Let us start by describing the framework in which we will study our PDEs. Let $(\Hil,\omega)$ be a separable symplectic Hilbert space, meaning a separable Hilbert space with an anti-symmetric bilinear map $\omega$ for which $i_\omega:\Hil\to\Hil^*$ is an isomorphism. Following \cite{kuksin95}, we fix a complete Darboux basis $\{e_n^\pm\}_{n\in \Z}$ in the sense that $\omega(e_i^+,e_j^-)=\delta_{ij}$. We define an anti-symmetric linear operator $J$ by $Je_n^\pm:=\mp e_n^\mp$ so that we can write the symplectic form as $\omega=\langle \cdot,J\cdot\rangle$ for some equivalent inner product on $\Hil$ which we fix from now on. Introducing the complex basis $z_n:=2^{-1/2}(e_n^++ie_n^-)$ for $n\in\Z$, the Hilbert space $\Hil$ can be identified with a subspace of the complexified Hilbert space $\Hil\otimes\C$ spanned by $z_n$, $n\in\Z$, where $J=i$ and $\omega=idz\wedge d\overline{z}$. 
% We can choose a Darboux basis $\{(p_n,q_n)\}_{i\in\Z_{\geq0}}$ for $\Hil$, so that we can write the linear symplectic form as $\omega=\sum_n dp_n\wedge dq_n$. There exists a compatible almost-complex structure $J$, which in these coordinates looks like $Jp_n=-q_n$ and $Jq_n=p_n$, so that the inner product on $\Hil$ satisfies $\langle\cdot,\cdot\rangle=\omega(\cdot,J\cdot)$.
For a smooth Hamiltonian function $H:\Hil\to\R$, the symplectic gradient $X_H$ is defined by $dH(\cdot)=\omega(X_H,\cdot)$ so that $X_H=J\nabla H$. We then write the Hamiltonian equation as
\begin{align}
\label{hamde}
\dot{u}=X_H(u)=J\nabla H(u).
\end{align}
In this paper, we consider general time-dependent Hamiltonian PDEs of the form
\[
\dot{u}=JAu+J\nabla F_t(u)
\]
with underlying Hamiltonian
\begin{align}
\label{hamform}
H_t(u)=\frac{1}{2}\langle Au,u\rangle+F_t(u)=:H_A(u)+F_t(u)
\end{align}
for some time-dependent and $T$-periodic nonlinearity $F_t$ defined in \Cref{admissible}, and quadratic term $H_A$ (also called the free term) defined by a linear, possibly unbounded, self-adjoint (differential) operator $A:\text{Dom}(A)\subset\Hil\to\Hil$ where the domain of $A$ is dense in $\Hil$. We will always assume the nonlinearity to be smooth in the time variable. We restrict to the case where the Hilbert space is a space of functions in one variable, which depends on the specific Hamiltonian PDE. See \cite{kuksin95}, \cite{berti} and \cite{digregorio} for examples of Hamiltonian PDEs. \\\par%The results in this paper can be generalized to functions with more space coordinates.

% We first define Hilbert scales, ollowing \cite{kuksinbook}. Let $\Hil_0$ be a Hilbert space with inner product $\langle\cdot,\cdot\rangle$ and basis $\{\nu_n|n\in\widetilde{\Z}\}$ with $\widetilde{\Z}\subset\Z^k$ a countable even (meaning $-\widetilde{\Z}=\widetilde{\Z}$) subset. Let $\{\theta_n|n\in\widetilde{\Z}\}$ be a sequence of positive real numbers such that $\theta_n=\theta_{-n}$ and $\theta_n\to\infty$ as $|n|\to\infty$. For $s\in\R$, define $\Hil_s$ to be the Hilbert space with basis $\{\nu_n\theta_n^{-s}|n\in\widetilde{\Z}\}$. When writing an element $u\in\Hil$ as $u=\sum \hat{u}(n)\nu_n$, we define the inner product and norm in $\Hil_s$ by
%\[
%\langle u,v\rangle_s^2=\sum |\hat{u}(n)\hat{v}(n)|\theta_n^{2s}.
%\]
%We call $\{\Hil_s\}$ a \textit{Hilbert scale}. One properties of Hilbert scale we mention here is that $\Hil_s$ is dense and embeds compactly in $\Hil_r$ for $s>r$.%, and $\Hil_s$ and $\Hil_{-s}$ are conjugate with respect to the inner product on $\Hil_0$, in the sense that for $u\in\Hil_s\cap\Hil_0$ we have
%\[
%\|u\|_s=\sup\left\{\langle u,v\rangle|v\in\Hil_{-s}\cap\Hil_0,\;\|v\|_{-s}=1\right\}.
%\]

Before we give some examples of Hamiltonian PDEs with Hamiltonian of the form \eqref{hamform}, let us first address the problem that the %differential operator $A$ of the 
free term $H_A$ is actually only densely defined when the differential operator $A$ is %a differential operator 
of positive order. However, the free flow $\phi_t^A=e^{tJA}$ is a linear unitary map and hence extends to a linear unitary map on the whole space $\Hil$. Indeed, we obtain an sc-Hamiltonian flow in the sense of \Cref{appendix}. Even though the existence of the free flow can be established, we have to be careful about the kind of nonlinearities we allow for the flow of the full Hamiltonian to still be guaranteed: even when $F_t$ is smooth, there need not exist a flow of $F_t$ due to compactness problems. The existence of the flow is a very delicate problem and a big amount of great work has been done in this direction. In this paper we avoid this problem by working with nonlinearities for which finite-dimensional approximations exist and are immediate. In particular, below we consider two important examples of Hamiltonian PDEs with regularizing nonlinearities (as in e.g. \cite{kuksin2015kam} or \cite{kuksin2016kam}). Such nonlinearities appear in nonlocal (or quasilocal) classical field theories, where fields have nonlocal interactions, and actually lead to a Hamiltonian partial \emph{integro-}differential equation. In contrast to local models, such nonlocal models, which in many cases model reality even better, are almost never integrable. The models we consider below can have arbitrary nonlocality, or quasilocality \cite{tomboulis}. For the relevance of nonlocal Hamiltonian PDE's see e.g. \cite{eringen}, \cite{haas} or \cite{zhong}.
\begin{example}[Nonlinear wave equation (NLW)]
\label{nlw1}
We write the nonlinear wave equation with regularizing nonlinearity as
\begin{align}
\label{nlweqn}
\ddot{\varphi}-\varphi_{xx}-\partial_1g_t(\varphi*\psi,x)*\psi-c_t=0,\qquad \varphi=\varphi(t,x)=\varphi(t,x+X),\; x\in S^1=\R/X\Z
\end{align}
with $\psi\in C^h(S^1)$ for $h>0$ and $g_{t+T}=g_t$ being bounded, smooth in both components and having bounded derivatives in the first component and $c_t=c_{t+T}\in C^h(S^1)$ denoting some time-dependent exterior potential \footnote{Some authors write NLW with an extra term $+mu$ with $m\geq0$ on the LHS. We choose to set $m=0$.}. A specific example of this would be a nonlocal sine-Gordon equation with exterior potential $$\ddot{\varphi}-\varphi_{xx}-a_t\sin(\varphi*\psi)*\psi-c_t=0.$$ This appears, for example, in a nonlocal Frenkel-Kontorova model with time- and space-periodic coefficient $a_t$ and exterior potential $c_t$. %This modification of the original sine-Gordon equation takes non-local interactions into account and our results imply the existence of a time-periodic solution for this equation with arbitrary non-locality. 
\par
The nonlinear wave equation is a Hamiltonian PDE on the Hilbert space $\Hil=L^2(S^1,\R)\times L^2(S^1,\R)$. It can be written as HPDE as
\[
\kolomtwee {\dot{\varphi}}{\dot{\pi}}=\kolomtwee {\pi}{\varphi_{xx}+\partial_1g_t(\varphi*\psi,x)*\psi+c_t}=J\nabla_{L^2}H_t(\varphi,\pi)
\]
with $J(\varphi,\pi)=(-\pi,\varphi)$ and with Hamiltonian
\[
H_t(\varphi,\pi)=\frac{1}{2}\int_0^X \left(-\varphi_x^2- \pi^2+2g_t(\varphi*\psi,x)+2c_t\varphi\right)\, dx.
\]
%A specific example of a non-linear wave equation is the generalized sine-Gordon equation
%\[
%\ddot{\varphi}-\varphi_{xx}-\partial_1g_t(\sin(\varphi*\psi),x)*\psi=0
%\]
%where now we only require that $g_t$ is smooth.\par
However, this Hilbert space on which NLW is modeled does not admit a complete Darboux basis which is compatible with the differential operator $A$. We will study a different structure in \Cref{nlw2}.
\end{example}\par
\begin{example}[Nonlinear Schrödinger equation (NLS)]
\label{nls1}
Consider the nonlinear Schrödinger equation with regularizing nonlinearity
\begin{align*}
%\label{nlseqn}
i\dot{u}+u_{xx}+\partial_1f_t\left(\vert u*\psi\vert^2,x\right)(u*\psi)*\psi=0,\qquad u=u(t,x)=u(t,x+X),\; x\in S^1=\R/X\Z
\end{align*}
with $\psi\in C^h(S^1)$ for $h>0$ and with $f_{t+T}=f_t$ smooth in both components. We also require that the map $\widetilde{f}_t:(s,x)\mapsto f_t(|s|^2,x)$ is bounded and has bounded derivatives in the first component. The Hilbert space is $L^2(S^1,\C)$. The Hamiltonian is given by
\[
H_t(u)=\frac{1}{2}\int_0^X\left(-\vert u_x\vert^2+f_t(\vert u*\psi\vert^2,x)\right)\,dx.
\]
This Hamiltonian PDE descends to an infinite-dimensional Hamiltonian system on projective Hilbert space.
\end{example}
%For more examples of HPDEs, see also \cite{digregorio}.\\\par
%We will consider arbitrary Hamiltonian PDEs which live on linear Hilbert space and impose boundary conditions in the form of Lagrangian subspaces. We define a Lagrangian subspace $L\subset\Hil$ as a linear subspace for which $L_n:=L\cap\R^{2n}$ is a Lagrangian subspace of $\R^{2n}$. The condition $u(t+1,x)=\phi_1^Au(t,x)$ can be translated to the condition that $u(0,x)\in L$ and $u(1,x)\in\phi_1^A(L)$. Since we would like to construct a Lagrangian Floer theory, we want the Lagrangians $\ell$ and $\phi_1^A(L)$ to intersect transversally. In particular, the origin should be the only intersection point. With $\omega$ as described above, the Lagrangians must be 1-dimensional linear subspaces of every 2-plane $\Span_\R\{e_n^\pm\}$. Since $\phi_1^A$ rotates every such 2-plane by an angle $\lambda_k$, this means that the Lagrangians which we can consider must consist of real linear subspaces of these 2-planes, as long as $A$ has no eigenvalues in $\Z\pi$. We define
%\begin{definition}[Admissible $(A,L)$]
%We call the pair $(A,L)$ of linear term and Lagrangian $L\subset\Hil$, as defined above, \textup{admissible} when $A$ has no eigenvalues in $\pi\Z$. 
%\end{definition}
%
%The small-divisor problem then appears when we let $n\to\infty$: then the angle of rotation will tend to zero (or a multiple of $2\pi$), making the intersection non-transversal. This will be addressed later.\\\par
See \cite{fabertnlse} for more details on this last example. In contrast to \cite{fabertnlse}, in this paper we will not focus on specific examples but rather consider nonlinear Hamiltonian PDEs with general nonlinearities on linear space.\par
Note that even though the nonlinearity is not local, it can be quasilocal in the sense that the smoothing kernel $\psi$ can have arbitrarily small support.%\\\par
%Note that in contrast to the finite-dimensional case, the smoothing kernel $\psi$ can have arbitrarily small support. \\\par
%In general we would not be able to ensure that the flow of the full Hamiltonian is well-defined on all of $\Hil$. In our examples, we resolve this existence problem by working with nonlinearities of convolution type, with smoothing kernel $\psi$ of class $C^s$. Then the flow of $F_t$ does exist when $s>0$. More precisely, we show that they define an sc-Hamiltonian flow which we define below. 
\section{Diophantineness condition}
\label{hilbertdiophantine}
Let us start with the free term of the Hamiltonian. Since $A$ is self-adjoint, we can diagonalize it. Here we have to make an assumption.
\begin{definition}
\label{admissibleA}
The differential operator $A$ of the free term $H_A$ is called \textup{admissible} when it is pure of degree $d\geq 1$ and there exists a complete Darboux basis $(e^{\pm}_n)$ of eigenvectors of the operator $A$ in the sense that $Ae^\pm_n=\lambda_ne^\pm_n$ with real eigenvalues of the form $\lambda_n=an^d$, $n\in\Z$ and $a\in\R_{>0}$. 
\end{definition}
From now on we assume that the operator $A$ is admissible. Note that the condition that $e_n^+$ and $e_n^-$ have the same eigenvalue and form a Darboux basis is equal to the statement that the commutator $[J,A]=0$. Our two main examples, see below, satisfy this condition and have eigenvalues of the form $\lambda_n=(2\pi/X)^dn^d$ with $X>0$ denoting the space periodicity. \emph{Despite the fact that our results apply to general symplectic Hilbert spaces $\Hil$, we will hence assume in what follows that $a=(2\pi/X)^d$.} So let us choose such a complete Darboux basis consisting of eigenvectors of $A$, so that $Ae^\pm_n=\lambda_ne^\pm_n$. Then the operator $JA$ is diagonal in the complex basis spanned by $z_{n}=2^{-1/2}(e_n^++ ie_n^-)$ with eigenvalues $i\lambda_n$ for $n\in\Z$. Following \cite{kuksin95}, we note that the flow maps of the free Hamiltonian $\phi_t^A=e^{tJA}$ define a family of linear symplectomorphisms on $\Hil$ which restrict to linear symplectomorphisms on the finite-dimensional subspaces $\C^{2k+1}=\Span_\C\{z_{n}\}_{n=-k}^k$. %Note that $z_{n}$ and $\overline{z}_{n}$ are linearly independent as functions. %Depending on the context, we will write either $z_n$ and $\overline{z}_n$ or $z_{\pm n}=2^{-1/2}(e_n^+\pm e_n^-)$. %Motivated by examples, we will write the eigenvalues as $\lambda_n=an^d$, where $a\in\R$ depends on the space period and $d\in\N$ is the order of the differential operator $A$, assuming it is not a linear combination of differential operators. 
The eigenvalues of the time-$T$ flow are $e^{ian^dT}$. %Before continuing, let us again consider some examples to motivate the above assumptions. \par
%\begin{example}
%\label{trivialexample}
%If the operator $A$ is of the form $A=a\Id$ with $a\in\R$, which in particular is of order $d=0$, then the Hamiltonian equation is
%\[
%\dot{u}=iau+X^F_t(u)
%\]
%with some nonlinearity $F_t$ with $h>0$ and some symplectic Hilbert space $\Hil$. The free flow is $\phi_t^A=e^{iat}$.
%\end{example}
Let us now consider the examples NLW and NLS from before (see \cite{kuksin95}).
\begin{example}[NLW]
\label{nlw2}
The nonlinear wave equation in one space dimension was modelled in \Cref{nlw1} on $\Hil=L^2(S^1,\R)\times L^2(S^1,\R)$. However, we want the Hilbert basis to be a complete Darboux basis of eigenvectors of the operator $A$. This forces us to choose $\Hil=W^{\frac{1}{2},2}(S^1,\R)\times W^{\frac{1}{2},2}(S^1,\R)$ with $S^1=\R/X\Z$. In this setting, we study equation \eqref{nlweqn} with the same nonlinearity. % being polynomial in th and satisfying
%\[
%\left|\partial_{u}^a\partial_t^bg(u;t,x)\right|\leq c_k(1+|u|)^{d_k}\qquad\text{for}\qquad a+b=k\geq 0
%\]
%with $c_k>0$ and $d_k\geq0$ bounded and independent of $t$. 
Now we write the operator $A$ as $A=\Diag(B,B)$ with $B=\sqrt{-\partial_x^2}$ and we write the nonlinear wave equation as
\[
\kolomtwee {\dot{\varphi}}{\dot{\pi}}=\kolomtwee {-B\pi}{B\varphi-B^{-1}\partial_1g_t(\varphi*\psi)*\psi-B^{-1}c_t}.
\] 
The inner product on $W^{\frac{1}{2},2}(S^1,\R)$ is
\[
\langle f,g\rangle=\frac{1}{2\pi}\int_0^XBf(x)g(x)dx
\]
and the Hamiltonian on $\Hil$ is given by
\[
H_t(\varphi,\pi):=\frac{1}{2}\left\langle A\kolomtwee\varphi\pi,\kolomtwee\varphi\pi\right\rangle+\frac{1}{2\pi}\int_0^Xg_t(\varphi*\psi,x)+c_t\varphi\;dx,
\]
where we extend the inner product componentwise to $W^{\frac{1}{2},2}(S^1,\R)\times W^{\frac{1}{2},2}(S^1,\R)$. The complete Darboux basis is then given by
\[
&e_n^+=\frac{1}{\sqrt{|n|}}(\xi_n(x),0),\qquad e_n^-=\frac{1}{\sqrt{|n|}}(0,\xi_n(x));\\
&\xi_n(x)=\begin{cases}\sqrt{2}\cos\left(\frac{2\pi n x}{X}\right)\quad n\leq 0\\
\sqrt{2}\sin\left(\frac{2\pi n x}{X}\right)\quad n>0
\end{cases}
\]
and the eigenvalues of $A$ are 
\[
Ae_n^\pm=\frac{2\pi n}{X}e_n^\pm
\]
so $\lambda_n=an^d$ with $a=2\pi/X$ and $d=1$. The Hilbert space can be identified with the subspace $\Span_\C\{z_n\}_{n\in\Z}$ of the complexified Hilbert space $\Hil\otimes\C$, with 
\[
z_n=\frac{1}{\sqrt{2|n|}}(\xi_n(x),i\xi_n(x))
\]
and the flow maps are 
\[
\phi_T^Az_n=e^{iT\frac{2\pi n}{X}}z_n.
\]
\end{example}
\begin{example}[NLS]
\label{nls2}
The Hilbert space for this PDE is $L^2(S^1,\C)$ with inner product which, when viewed as a real vector space with inner product
\[
\langle f,g\rangle=\re\frac{1}{X}\int_0^X f\overline{g}dx
\]
has complete Darboux basis given by 
\[
e_n^+=e_n,\qquad e_n^-=-ie_n
\]
where $\{e_n\}_{n\in\Z}$ is the complete system of eigenfunctions of $-\partial_x^2$ given by
\[
e_n(x)=e^{i\frac{2\pi nx}{X}}.
\]
These have eigenvalues $\lambda_{n}=(2\pi n/X)^2$. We can identify this real Hilbert space $(\Hil,J=i)$ with the complex Hilbert space spanned by $z_n:x\mapsto \sqrt{2}e^{i\frac{2\pi nx}{X}}$ with $n\in\Z$. The time-$T$ flow of the free part of the Hamiltonian is 
\[
\phi_T^Az_{ n}=e^{iT\left(\frac{2\pi n }{X}\right)^2}z_{ n}.
\]
Writing the eigenvalues in the form suggested above, we get $\lambda_n=an^d$ where $a=(2\pi/X)^2$ and $d=2$. %See \cite{fabertnlse} for more details.
\end{example}
%If we set 
%\[
%e_n^+:=\begin{cases}\cos(nx)\quad &n\leq0\\
%i\sin(nx)\quad &n>0\end{cases},\qquad
%e_n^-:=-ie_{-n}^+=\begin{cases} \sin(nx)\quad &n<0\\
%-i\cos(nx)\quad &n\geq0\end{cases}
%\]
%then
%\[
%z_n^\pm:=e_n^+\pm ie_n^-=\begin{cases} \cos(nx)\pm i\sin(nx)\quad &n<0\\
%i\sin(nx)\pm\cos(nx)\quad &n>0\\
%1\pm 1\quad &n=0\end{cases}
%\]
%and $J$ is defined by $Je_n^\pm=\mp e_n^\mp$, which can alternatively be written as $J=\frac{1}{n}\partial_x$.\\\par
%
Here we already catch a glimpse of what will be a problem we need to address, which does not appear in the finite-dimensional case: in order to apply the machinery of Floer theory in our infinite-dimensional situation, we would like to ensure that the system is nondegenerate, i.e. that the time-$T$ flow map has no eigenvalue equal to $1$. This in turn means that $aT/2\pi$ must not be rational. Compare this condition for \Cref{nlw2} with \cite{rabinowitz}, who proved existence of forced time-periodic solutions when the number $aT/(2\pi)=T/X$ \textit{is} rational. For general eigenvalues $\lambda_n=an^d$ we need $aT/2\pi=T/X^d\cdot (2\pi)^{d-1}$ to be irrational, where we recall  that $a=(2\pi)^d/X^d$ implicitly depends on the space period $X$. However, even if these numbers are irrational, we are faced with the problem that a subsequence of the eigenvalues of the flow will always converge to $1$. Let us illustrate this problem somewhat: to prove the existence of a solution to the nonlinear PDE, we will have to assume that the time-$T$ flow of the free Hamiltonian $\phi_T^A$ has only one fixed point, or, alternatively, that the only solution to the free Hamiltonian equation
\begin{align}
\label{freehpde}
\dot{u}=JAu,\qquad u(0)=u(T)
\end{align}
is $u\equiv0$. When $aT/2\pi$ is irrational, this forces the only solution to \eqref{freehpde} to be $u\equiv0$: if $u_0$ is a fixed point of the time-$T$ free flow, then expanding $u_0$ as $\sum \hat{u}_0(n)z_n$ shows that 
\[
\phi^A_Tu_0=\phi^A_T\sum_{n}\hat{u}_0(n)z_{n}=\sum_{n}e^{iaTn^d}\hat{u}_0(n)z_{n}.
\]
So as long as $aT/2\pi$ is irrational, for any $n$ there are no eigenvalues equal to one. In the limit, however, this is not guaranteed: there could be a subsequence of $(e^{iT\lambda_n})_n$ converging to $1$. This is an instance of the small divisor problem. To solve this problem, we need to control the way in which the eigenvalues (or a subsequence of them) converge to $1$. The essence of our approach is that we should not be able to approximate the irrational number $aT/2\pi$ too well by rational ones. More formally, we make the following definition.
\begin{definition}
\label{admissibleXT}
We call the pair of time and space periods $(T,X)\in\R_{>0}\times\R_{>0}$ \textup{admissible} when the number $aT/2\pi=(2\pi)^{d-1}\cdot T/X^{d-1}$ is Diophantine.
\end{definition}  
%Again, we remark that when $d=0$ it is sufficient that the number $aT$ is not in $2\pi\Z$, in which case we call $(T,X)$ admissible.\par
In particular, the Diophantine number $aT/2\pi$ has finite \textit{irrationality measure}. Let us explain this statement: every real number can be approximated by a continued fraction and this gives a measure of how good a real number can be approximated by rationals. For all $\sigma\in\R$ there exists $p/q\in\Q$ such that
\[
\left|\sigma-\frac{p}{q}\right|<\frac{1}{q^2}.
\]
The irrationality measure gives is a measure of how good this approximation can be. It is defined as the infimum of the set of real numbers $\rho$ for which
\[
\frac{c}{q^\rho}<\left|\sigma-\frac{p}{q}\right|<\frac{1}{q^2}
\]
holds for all $p/q\in\Q$ with some fixed $c>0$, and is usually denoted by $r$. In particular, it is at least $2$. It turns out that the set of numbers of irrationality measure $2$ (and hence of Diophantine numbers) has full Lebesgue measure \cite[theorem E.3]{bugeaud}  For $\pi$, it is shown in \cite{salikhov} that $r<8$. This is used, for example, in \cite{fabertnlse} to prove a statement similar our main theorem for the NLSE on projective Hilbert space. In what follows, by \emph{generic time period} $T$ we will mean $T$ for which $aT/2\pi$ has irrationality measure $r=2$.\\\par
Before turning to the nonlinearity, let us first give an example which shows that the Diophantineness condition (and subsequent regularity conditions for the nonlinearity) are really necessary.
\begin{example}[Counterexample]
\label{counterexample}
Consider the linear wave equation with exterior potential
\begin{align}
\label{nonhomwave}
\ddot{\varphi}=\varphi_{xx}+c_t,\qquad c_{t+T}=c_t.
\end{align}
Let us write $\varphi$ and $c$ as Fourier series 
\begin{align*}
\varphi&=\sum_{p,n\in\Z}\hat{\varphi}(p,n)e^{\frac{2\pi i pt}{T}}e^{\frac{2\pi i nx}{X}} \nonumber\\
c&=\sum_{p,n\in\Z}\hat{c}(p,n)e^{\frac{2\pi i pt}{T}}e^{\frac{2\pi i nx}{X}}
%\label{counterfourier}
\end{align*}
with $\overline{\hat{\varphi}(p,n)}=\hat{\varphi}(-p,-n)$ and $\overline{\hat{c}(p,n)}=\hat{c}(-p,-n)$ such that the functions are real-valued. Term-wise, \eqref{nonhomwave} becomes
\[
\hat{\varphi}(p,n)\pa{\pa{\frac{2\pi n}{X}}^2-\pa{\frac{2\pi p}{T}}^2}e^{\frac{2\pi i pt}{T}}e^{\frac{2\pi i nx}{X}}=\hat{c}(p,n)e^{\frac{2\pi i pt}{T}}e^{\frac{2\pi i nx}{X}}
\]
or
\begin{align*}
%\label{countereqn1}
\hat{\varphi}(p,n)=\frac{\hat{c}(p,n)\frac{T^2}{(2\pi n)^2}}{\pa{\frac{T}{X}-\frac{p}{n}}\pa{\frac{T}{X}+\frac{p}{n}}}.
\end{align*}
When $T/X$ is not Diophantine, there is a subsequence $(p',n')\subset(p,n)_{p,n\in\Z}$ for which the denominator in the expression for $\hat{\varphi}(p',n')$ goes to zero exponentially fast. If we define the exterior potential $c_t$ by
\[
\hat{c}(p,n):=\begin{cases}
\frac{(2\pi n)^2}{T^2}\pa{\frac{T}{X}-\frac{p}{n}}\pa{\frac{T}{X}+\frac{p}{n}}\qquad &(p,n)=(p',n')\\
0\qquad &\text{otherwise}
\end{cases}
\]
then $c_t$ is smooth, but $\hat{\varphi}(p',n')$ is constantly $1$, so that there is no solution. Writing the exterior potential in the Hamiltonian for NLW as in \Cref{nlw2} as $\langle (c_t,0),\cdot\rangle$, we see that $c_t$ should have some minimal regularity, depending on the irrationality measure of $T/X$, to ensure the existence of a solution.
\end{example}

%There can be a subsequence $(n_k)_{k\in\N}\subset(n)_{n\in\N}$ for which $(e^{iT\lambda_{k_n}})_{n\in\N}$ converges to one. However, for any $n\in\N$, the eigenvalues $e^{iT\lambda_n}\neq1$ as long as $aT/2\pi$ is irrational. So if $u_0$ is a fixed point, then for any $n\in\N$ the coefficient $\hat{u}_0(n)$ should be zero, which shows that the only fixed point $u_0$ to $\phi_T^A$ is $u_0\equiv0$. This happens when $(T,X)$ is admissible, which is a generic condition. See also the discussion following \Cref{admissible}.\\\par

\section{$A$-admissible and weakly $A$-admissible nonlinearities}
\label{admissiblenonlin}
In order to deal with the asymptotic degeneracy caused by the small divisor problem, our key idea is as follows: in order for Gromov-Floer compactness to hold, we want to assume that the nonlinearity can be approximated by finite-dimensional ones better than the eigenvalues of the time-$T$ flow of the free Hamiltonian approach $1$. This puts restrictions on the regularity of the nonlinearity. To explain this, consider the following: expanding $u\in\Hil$ as $u=\sum_n\hat{u}(n)z_n$ we let $u^k$ denote the restriction to $\C^{2k+1}=\Span_\C\{z_{n}\}_{n=-k}^k\subset\Hil$ given by
\[
u^k:=\sum_{n=-k}^k\hat{u}(n)z_{n}.
\]
The finite-dimensional restriction $F_t^k$ of the nonlinearity is then
\[
F_t^k(u):=F_t(u^k)
\]
and we write $X^{F,k}_t$ for the symplectic gradient of this finite-dimensional restriction. Then the flow $\phi_t^{F,k}$ of the restricted Hamiltonian $F_t^k$ restricts to the finite-dimensional subspace $\C^{2k+1}\subset \Hil$.\\\par

To formalize the idea that we need some minimal regularity for the nonlinearity to make our methods work, we start by introducing Hilbert scales in the sense of \cite{kuksinbook}: working in the complex Hilbert space spanned by $z_n$, $n\in\Z$, where we can identify $J$ with $i$, our separable symplectic Hilbert space $\Hil$ is isometrically isomorphic to $\ell^2(\Z,\C)$ via the complete basis $\{z_n\}_{n\in \Z}$ by 
\[
\Hil\ni u=\sum_{n}\hat{u}(n)z_n\mapsto \left(\hat{u}(n)\right)_{n\in\Z}
\]
where the sum is understood to be over $n\in\Z$. We will write $\Hil_0=\Hil$ and define $\Hil_1$ to be the (dense) subspace of $\Hil_0$ consisting of those $u=\sum_n\hat{u}(n)z_n$ for which $\sum_n\hat{u}(n)nz_n$ is in $\Hil_0$. We define $\ell^{2,1}$ to be the image of $\Hil_1$ under the isomorphism between $\Hil_0$ and $\ell^2$ described above. More generally, we define
\[
\Hil_h:=\left\{u\in\Hil_0\left|\sum_n\hat{u}(n)n^hz_n\in\Hil_0\right.\right\},\quad\ell^{2,h}:=\left\{\left(\hat{u}(n)\right)_{n\in\Z}\in\ell^2\left| \left(\hat{u}(n)n^h\right)_{n\in\Z}\in\ell^2\right.\right\}
\]
for $h\geq0$. Similarly, we define the sequence space $\ell^{2,-h}$ as
\[
\ell^{2,-h}:=\left\{\left(\hat{u}(n)\right)_{n\in\Z}\left|\left(\hat{u}(n)n^{-h}\right)\in\ell^2\right.\right\}
\]
for $h>0$ as the space of possibly diverging sequences and, using the Darboux basis, we identify this with the subspace $\Hil_{-h}$ of the space of tempered distributions. The totality $(\Hil_h)_{h\in\R}$ is also known as a Hilbert scale. We let $\Hil_\infty=\cap\Hil_h$ and $\Hil_{-\infty}=\cup\Hil_h$. Note that $\Hil_h$ is dense and embeds compactly in $\Hil_i$ when $h>i$. \\\par
%Our symplectic Hilbert space $\Hil$ is isometrically isomorphic to $\ell^2$ via the complete Darboux basis\footnote{Since $\Z\cong\N$ we make no distinction and simply use $\N$.} $\{z_n\}_{n\in \N}$, by 
%\[
%\Hil\ni u=\sum_{n}\hat{u}(n)z_n\mapsto \left(\hat{u}(n)\right)_{n\in\Z}.
%\]
%We will write $\Hil_0=\Hil$ and define $\Hil_1$ to be the (dense) subspace of $\Hil_0$ consisting of those $u=\sum_n\hat{u}(n)z_n$ for which $\sum_n\hat{u}(n)nz_n$ is in $\Hil_0$. We define $\ell^{2,1}$ to be the image of $\Hil_1$ under the isomorphism between $\Hil_0$ and $\ell^2$ described above. More generally, we define
%\[
%\Hil_k:=\left\{u\in\Hil_0\left|\sum_n\hat{u}(n)n^kz_n\in\Hil_0\right.\right\},\qquad\ell^{2,k}:=\left\{\left(\hat{u}(n)\right)_{n\in\N}\in\ell^2\left| \left(\hat{u}(n)n^k\right)_{n\in\N}\in\ell^2\right.\right\}
%\]
%With this in place, we say make the following definition.

\begin{definition}
\label{sregularizing}
A map $F_t:\Hil_0=\Hil\to\R$ is called $h$-\textup{regularizing} if it extends to a smooth map
\[
F_t:\Hil_{-h}\to\R,
\]
and it is called $\infty$-regularizing when it is $h$-regularizing for all $h\in\N$. 
\end{definition} 
Note that when $F_t$ is $h$-regularizing, then the differential defines a map
\[
dF_t:\Hil_{-h}\to\left(\Hil_{-h}\right)^*\cong \Hil_h
\]
and so, in particular, for the gradient (with respect to the inner product on $\Hil$) it holds that 
\[
\nabla F_t:\Hil_0\subset\Hil_{-h}\to\Hil_h.
\]
Note that this latter property can also be stated in terms of Kuksin's Hilbert scale theory by saying that $\nabla F_t$ is a scale morphism of the Hilbert scale $(\Hil_k)_k$ order $-h$. \\\par

\begin{lemma}
\label{lemma1}
Assume the nonlinearity is $h$-regularizing with $h>0$ such that the extended map $F_t:\Hil_{-h}\to\R$ has bounded $C^\alpha$-norms for all $\alpha\in\N$. Then $\nabla F^k$ converges to $\nabla F$ uniformly with all derivatives when viewed as maps into $\Hil$.  Furthermore, when expanding $\nabla F_t$ into a Fourier series
\[
\nabla F_t(u)=\sum_{n\in\Z}\widehat{\nabla F_t(u)}(n)z_n
\]
we have %$\left\|\nabla^k_{\perp} F\right\|_{C^0}=o(k^{-h})$.
$\widehat{\nabla F_t(u)}(n)=o(|n|^{-h})$.
%, where in the first case the supremum norm is taken over the set $U_R^k\subset\Hil$ of all points in $\Hil$ which have distance less than $R$ from $\C^k$ for any fixed $R>0$ .
%
%
% the boundedness assumption $\|F\|_{C^\alpha}<C_\alpha$ and has bounded support $F_t(u)=0$ when $|u|>R$. Then the flow $\phi_t^{F,k}$ of the restricted Hamiltonian $F_t^k$ restricts to the finite dimensional subspace $\C^k\subset \Hil$. The sequence $(F^k)_k$ converges uniformly in $k$ and with all derivatives to $F$. In particular, the sequences $X^{F,k}$ and $\phi_T^{F,k}$ of vector fields and time-$T$ flows converge uniformly as well.
\end{lemma}
\begin{proof}
%We have
%\[
%\left\langle \nabla F_t^k\left(u\right),v\right\rangle%&=\left(dF_t^k\right)_u\left(v\right)\\
%%&=\left(d\left(F_t\circ\pi_k\right)\right)_u\left(v\right)\\
%%&=\left(\left(dF_t\right)_{\pi_k\left(u\right)}\circ \left(d\pi_k\right)_u\right)(v)\\
%%&=\left(dF_t\right)_{\pi_k\left(u\right)}\left(\pi_k\left(v\right)\right)\\
%%&=\left\langle \nabla F_t\left(\pi_k\left(u\right)\right),\pi_k\left(v\right)\right\rangle\\
%&=\left\langle \pi_k\nabla F_t\left(\pi_k\left(u\right)\right),v\right\rangle
%\]
%so that the vector field $X_t^{F,k}$ and flow $\phi_t^{F,k}$ restrict to $\C^k$. For the convergence, note that
For the first statement note first that the boundedness of the $C^1$-norm of $F_t:\Hil_{-h}\to\R$ implies that the $C^0$-norm of  $\nabla F_t:\Hil(\subset\Hil_{-h})\to\Hil_h$ is bounded which, together with the compact embedding $\Hil_h\subset\Hil$, yields that $\nabla F^k(u)\to\nabla F(u)$ with respect to the $\Hil$-norm uniformly in $u\in\Hil$ as $k\to\infty$. Going beyond, note that the boundedness of the $C^{\alpha}$-norm of $F_t:\Hil_{-h}\to\R$ yields a uniform bound for the higher derivatives $\nabla^\alpha F_t(u)\in\Hil_h^{\otimes\alpha}$ for all $u\in\Hil$ which again implies that $\nabla^{\alpha}F^k_t(u)\to\nabla^{\alpha} F_t(u)$ with respect to the $\Hil^{\otimes\alpha}$-norm uniformly in $u\in\Hil$ as $k\to\infty$. Because $F_t$ satisfies the regularity assumption $\nabla F_t:\Hil\to\Hil_h$, the coefficients in the expansion of $\nabla F_t(u)$ satisfy $|n|^h\widehat{\nabla F_t(u)}(n)=o(1)$ and hence $\widehat{\nabla F_t(u)}(n)=o(|n|^{-h})$. \end{proof}
%: this follows from the fact that the coefficients $\widehat{\nabla F_t(u)}(n)$ are square summable with respect to the $\Hil_h$-norm, i.e. $|n|^{2h}\abs{\widehat{\nabla F_t(u)}(n)}^2=o(n^{-1})$ and hence $\abs{\widehat{\nabla F_t(u)}(n)}=o(|n|^{-h-\frac{1}{2}})$.
%Hence the $C^0$-norm of $\nabla^k_\perp F$ satisfies 
%\[
%\left\|\nabla_\perp^kF\right\|_{C^0}=o(k^{-h}).
%\] For the statement about the convergence of $F^k$ to $F$, observe that we have
%\[
%\left|F_t(u)-F_t^k(u)\right|&=\left|F_t(u)-F_t(u^k)\right|\\
%&\leq\left\|\nabla_\perp^kF\right\|_{C^0} |u-u^k|\\
%&\leq\left\|\nabla_\perp^kF\right\|_{C^0} R%\left|X_t^G(u)-X_t^{F,k}(u)\right|&=\left|\nabla F_t(u)-\nabla F_t^k(u)\right|_\Hil\\
%&=\left|\nabla F_t(u)-\nabla F_t(u_k)\right|_\Hil\\
%&\leq \left\|\nabla F_t-\nabla^k F_t\right\|_\infty R\\
%&=:\left\|\nabla^k_\perp F_t\right\|_\infty R
%%%&\leq \left\| F_t\right\|_{C^2(\Hil_{-s},\R)}C
%&\leq \left\|\nabla F_t-\nabla F_t^k\right\|_\infty\\
%&\leq\left\| F_t-F_t^k\right\|_{C^1}\\
%&=:f^1(k).
%\]

%Note the importance of the assumption $\| F\|_{C^\alpha}<\infty$: even though the finite-dimensional nonlinearities $F_t^k$ may have compact support, a ball in $\Hil$ is only bounded and not compact. 

In order to be able to use the results from Floer and symplectic homology for open sets in finite dimensions as in \cite{fh1}, \cite{oanceasurvey}, \cite{wendl}, we need a sequence of finite-dimensional Hamiltonians which converges in the proper sense to our infinite-dimensional one as above. %As we show above, these can be constructed for bounded $s$-regularizing nonlinearities with bounded support. 
To ensure that such an approximating sequence exists and that our methods work, we impose the following restrictions on the nonlinearity.
%To be more general, though, we assume that for our nonlinearities such an approximating sequence exist. More precisely, we impose the following restrictions on the nonlinearity. 
\begin{definition}
\label{admissible}
A nonlinearity $F_t:\Hil\to\R$ %of the form \eqref{nonlinform}
is called \textup{$A$-admissible} if $A$ is admissible and $F_t:\Hil\to\R$ satisfies the following conditions:
\begin{enumerate}
\item $F_t$ is $T$-periodic with $(T,X)$ admissible. 
\item The nonlinearity is $h$-regularizing with $h> dr$. Here $r$ is the irrationality measure of $aT/2\pi$ and $d\geq 1$ the order of the differential operator $A$. 
\item The extended map $F_t:\Hil_{-h}\to\R$ has bounded $C^\alpha$-norms for all $\alpha$.
%The nonlinearity $F_t$ can be approximated by its finite-dimensional restrictions $F_t^k$ uniformly in $k$ sufficiently well, that is, 
%\[
%\left\| F-F^k\right\|_{C^\alpha(\Hil,\R)}=o(k^{-s})
%\]

%for all $\alpha\in\N$. 
%is a map 
%\[
%F_t:\Hil_{-s}\stackrel{C^\infty}{\longrightarrow}\R
%\]
%which is smooth (as map from $\Hil_{-s}$ to $\R$) 
%with $s>s_0:=2+4d(r-1)$. Here $r$ is the irrationality measure of $aT/2\pi$ and $d$ the order of the differential operator $A$. %That is, $F_t$ is $s$-regularizing with $s>8+4d(r-1)+\frac{1}{2}$.

%\item There are bounds
%\[
%\left\| F\right\|_{C^k(\Hil_{-s},\R)}<C_k.
%\]

%\item The nonlinearity is sub-quadratic in the first variable:
%\[
%\frac{\left| \nabla F_t(u)\right|}{|u|}%=\frac{|\partial_1f_t((u*\psi)(x),x)*\psi|}{|u|}
%\longrightarrow0,\qquad |u|\to\infty
%\]
%uniformly in $t\in S^1$.
%%%\item The only solution to the Hamiltonian equation for the free Hamiltonian,
%%%\[
%%%\dot{u}=JAu,\qquad u(0)=u(T)
%%%\]
%%%is $u\equiv0$.
%\item For all $t\in\R$, the $C^2$-norm $c_n(t):=\left\| F_t^n\right\|_{C^2}$ is finite:
%\[
%c_n:=\sup_{t\in[0,T]}c_n(t)<\infty;
%\]
%and
%\[
%\left|\partial_t\nabla F_t(u)\right|%=|\partial_t\partial_1f_t((u*\psi)(x),x)*\psi|
%\leq c_n(1+|u|)
%\]
%for all $t\in S^1$ and $u\in\C^n$.
\item $F_t$ has bounded support, in the sense that for every $k\in\N$ there exists $R_k>0$ such that $F_t(u)=0$ for all $u\in\Hil$ with $|u^k|>R_k$.
\end{enumerate}
$F_t$ is called \emph{weakly $A$-admissible} when there exists $t$-dependent $c_t=c_{t+T}\in\Hil_h$ such that $u\mapsto F_t(u)-\langle c_t,u\rangle$ satisfies 1., 2., and 3. 
%An $A$-admissible nonlinearity $F_t$ is called \textup{strongly} $A$-admissible when $s>8+4d(r-1)+\frac{1}{2}+4(d+1)$. 
\end{definition}
%The first condition will be needed for proving that the higher-dimensional contributions are under control. The second and third condition are needed to prove that Floer curves stay
%
%The fourth condition will be needed for the $C^0$-bounds, the second, third and fifth allow us to prove a maximum principle, to ensure that the Floer curves stay inside a bounded region. Condition $4$ will be neede for the maximum principle as well, but can be improved on. The first condition is required to prove certain bounds and convergences in the finite-dimensional approximations. Compare the second condition also with condition $(f_3)(i)$ at the beginning of \cite{rabinowitz}\\\par
%Throughout this paper all $C^\alpha$-norms are with respect to $u\in\Hil$ (or the appropriate subspace) as well as $t\in\R$. 
We stress that the notion of (weakly) $A$-admissibility depends on the operator $A$ because the irrationality measure of the number $aT/2\pi$ associated to the eigenvalues of $A$, as well as the order of the PDE, dictate what regularity we need for the nonlinearity. Observe that the Diophantineness condition is generic in the sense that Diophantine numbers have full Lebesgue measure. Note, though, that the Diophantineness condition should rather be thought of as a condition on the time period, rather than on the eigenvalues of $A$: we start with a Hamiltonian PDE and this condition restricts what time periods the solutions can have. \\\par %Our results do not hold when the number $aT/2\pi$ is not Diophantine (for example, when it is rational). 

In order to explain the relation between $A$-admissible and weakly $A$-admissible nonlinearities, we prove the following  
\begin{proposition}
\label{cutoffprop}
Let $\widetilde{F}_t$ be a weakly $A$-admissible nonlinearity. Then 
\[
F_t(u):=\chi(|u|_{-h}^2)\widetilde{F}_t(u)
\]
with $h$ as in \Cref{admissible} condition 2, and where $\chi$ a smooth cut-off function with $\supp(\chi)\subseteq[0,R]$ for some $R>0$, is $A$-admissible.
\end{proposition} 
\begin{proof}
%Note that $|u*\psi|_0=|u|_{-s}$ since $\psi\in C^s(S^1)$ so that $F_t$ is indeed a map $F_t:\Hil_{-s}\to\R$. 
The first condition is immediate. In order to see that $F_t$ satisfies conditions $2$ and $3$, note that for every $c_t\in\Hil_h$ the map $u\mapsto\chi(|u|_{-h}^2)\langle c_t,u\rangle$ satisfies conditions $2$ and $3$ as $\langle c_t,u\rangle\leq |c_t|_h |u|_{-h}$. To establish condition $4$, let $R_k:=k^hR^{1/2}$ so that when the finite-dimensional restriction $u^k$ of $u$ satisfies $|u^k|_0>R_k$, then
\[
|u|^2_{-h}\geq|u^k|^2_{-h}=\sum_{n=0}^k\left|\hat{u}(\pm n)\right|^2n^{-2h}>k^{-2h}\sum_{n=0}^k\left|\hat{u}(\pm n)\right|^2>R
\]
so that $F_t(u)=0$.
\end{proof}

By the above proposition it hence suffices to find examples of weakly $A$-admissible nonlinearities. 

\begin{example} The nonlinearities from \Cref{nlw2} and from \Cref{nls1} are weakly A-admissible as long as $(X,T)$ is admissible and $h> dr$: when $u,\varphi\in\Hil_{-h}$ and $\psi\in C^h$, then $u*\psi,\varphi*\psi\in C^0$ for both examples. Together with the smoothness of $\widetilde{f}_t$ and $g_t$ it follows that the maps $x\mapsto\partial_1^{\alpha}\widetilde{f}_t((u*\psi)(x),x)$, $x\mapsto\partial_1^{\alpha}g_t((\varphi*\psi)(x),x)$ are continuous and hence (square-) integrable over $\R/X\Z$ for all $\alpha\in\N$. Altogether this is sufficient to prove that $F_t$ is smooth as a map from $\Hil_{-h}$ to $\R$ in both examples. Since $\widetilde{f}_t$ and $g_t$ have uniformly bounded derivatives, the maps $x\mapsto\partial_1^{\alpha}\widetilde{f}_t((u*\psi)(x),x)$ and $x\mapsto\partial_1^{\alpha}g_t((\varphi*\psi)(x),x)$ are uniformly bounded with respect to $u,\varphi$. But this implies that $\nabla^\alpha F_t$ is uniformly bounded for $\alpha=1,2,\ldots$; for $\alpha=0$ in \Cref{nlw2} only after subtracting a linear term as allowed in \Cref{admissible}. \end{example}
%Since $\widetilde{f}_t$ and $g_t$ are bounded, it follows that $F_t$ and all its derivatives are uniformly bounded. To make them $A$-admissible we compose them with a cut-off function:

%To give an example of an $A$-admissible nonlinearity, let $\widetilde{F}_t$ be the nonlinearity $\widetilde{f}_t$ from \Cref{nls1} or $g_t$ from \Cref{nlw2} and assume that $g_t$ is bounded and has bounded derivatives (we already assumed this for $\widetilde{f}_t$ in \Cref{nls1}). Assume that $\widetilde{F}_t$ is $T$-periodic with $(X,T)$ admissible. If we define $F_t$ by 
%\[
%F_t(u):=\chi(|u*\psi|^2)\widetilde{F}_t(u)
%\]
%where $\chi$ is a smooth cut-off function with support in $[0,R]$, then $F_t$ is $A$-admissible if $s>s_0$.\par

In the \Cref{mainthm} stated in \Cref{mainthmsection} we prove the existence of a Floer curve together with the existence of a periodic solution only for a Hamiltonian with $A$-admissible nonlinearity: in order to employ the maximum principle for proving compactness of the relevant moduli space of pseudoholomorphic curves, we do have to make the technical assumption concerning the support of the nonlinearity. In \Cref{subqsection}, however, we prove that the existence of a forced time periodic solution is still guaranteed when the nonlinearity is weakly $A$-admissible instead of $A$-admissible. 

\section{Main theorem}
\label{mainthmsection}
Before stating the main theorem, let us rewrite the setting a little: when the Hamiltonian $H_t$ is a sum of two terms $H_A$ and $F_t$, then the flows of $H_t$ and of $H_A$ and $F_t$ are related via
\begin{align*}
%\label{phiform}
\phi^{H_t}=\phi^{H_A+F_t}=\phi^{H_A\#G_t}=\phi^A\circ\phi^{G_t}
\end{align*}
where $G_t:=F_t\circ\phi^A_t$ and where $(H_A\#f)_t:=H_A+f_t\circ\phi^A_{-t}$ for any function $f_t$. We will work with $\phi_t^A$ and $G_t$ rather than with $H_t=H_A+F_t$ because $H_A$ (and hence $H_t$) is only densely defined, whereas the flow $\phi_t^A$ is a symplectomorphism which is defined on the whole of $\Hil$. Also $G_t$ turns out to have sufficiently nice properties, see \Cref{lemma4} where we show that even though $\phi_t^A$ is only differentiable on dense subspaces, $G_t$ is at least four times continuously differentiable in $t$.\par
Going back, we see that $T$-periodic solutions of \eqref{hamde} are in one-to-one correspondence with $u$ satisfying
\begin{align}
\label{phiform}
\dot{u}=X_t^G(u),\qquad u(t+T)=\phi_{-T}^A(u(t)).
\end{align}
We call such solutions $\phi_T^A$-periodic. We will prove existence of $\phi_T^A$-periodic solutions, which by the above correspondence implies existence of a true $T$-periodic solution. From now on we will use $G_t$ as in \eqref{phiform} instead of $F_t$ and say that $G_t$ is $A$-admissible when $F_t$ is. Note that the norms of $F_t$ and $G_t$ coincide for fixed $t$ because the free flow is unitary. Recall that when $G_t\equiv0$, the only solution to the PDE is $u\equiv0$. \\\par

Let $\varphi\in C^\infty(\R)$ be a cut-off function specified by
\[
\varphi(s)=0\;\mathrm{for}\;s\leq-1;\qquad\varphi(s)=1\;\mathrm{for}\;s\geq0;\qquad0\leq\varphi'(s)\leq2.
\]
\begin{maintheorem}
\label{mainthm}
\sloppy For a Hamiltonian PDE with $A$-admissible nonlinearity $G_t$ there exists a {$(\floor{h/d}-1)$-times} differentiable map $\widetilde{u}:\R\times \R\to\Hil_{h-d(r-1)-1/2}\subset\Hil$ for $h> dr$, called a \emph{Floer curve}, which satisfies the Floer equation and $\phi_T^A$-periodicity condition
\begin{align}
\label{floereqn}
%\widetilde{\partial}_{H_t^s}\widetilde{u}:=
\overline{\partial}\widetilde{u}+\varphi(s)\nabla G_t(\widetilde{u})=0,\qquad\widetilde{u}(s,t+T)=\phi_{-T}^A\widetilde{u}(s,t)
\end{align}
where $\overline{\partial}=\partial_s+i\partial_t$. The Floer curve $\widetilde{u}$ connects $u_0\equiv 0$ with a (weak) solution $u_1(t)$ of the nonlinear Hamiltonian PDE \eqref{phiform} and hence of \eqref{hamde}, in the sense that there exists sequences $s_n^\pm\in\R$ with $s_n^\pm\to\pm\infty$ as $n\to\infty$ such that
\begin{align*}
%\label{asymptoticcondition}
\lim_{n\to\infty}\widetilde{u}(s_n^-,t)=0,\qquad\lim_{n\to\infty}\widetilde{u}(s_n^+,t)=u_1(t)
\end{align*}
in the $C^{\floor{h/d}-1}$-sense. In particular, when the nonlinearity is $\infty$-regularizing, then both the Floer curve and the periodic orbit are smooth in all variables $s$, $t$ and $x$. 
%The solution $u_1(t)=u_1(t,x)$ to \eqref{phiform} is periodic in both variables as $u_1(t+T,x)=\phi_T^Au_1(t,x)$ and $u(t,x)=u_1(t,x+X)$.
\end{maintheorem}
Note that we call $u_1$ a weak solution, since $h-d(r-1)-\frac{1}{2}>d-\frac{1}{2}$ might not be large enough to guarantee that $u_1$ is also a solution in the classical sense. Here and after we continue to identify $\Hil$ with a subspace of the complexified Hilbert space spanned by $z_n$ for $n\in\Z$, and write $i$ instead of $J$. %because we are working in the complex eigenbasis for the operator $JA$, where $J$ is multiplication by $i$ in each complex plane $\Span_\C\{z_n\}$. 
We emphasize that we are using the setup of Floer homology for general symplectomorphisms from \cite{salamondostoglou} because even though the Hamiltonian $H_A$ is only densely defined, its flow $\phi_t^A$ is an everywhere defined symplectomorphism. To use this setup, we use the fact that $(\phi_{-T}^A)_*i=i$. \\\par
To go back from $G_t$ to $F_t$ and obtain a true $T$-periodic Floer curve for the Hamiltonian $H_t=H_A+F_t$, we define $\doubletilde{u}(s,t):=\phi^A_t\wt{u}(s,t)$ for $(s,t)\in\R\times\R$. It immediately follows that \eqref{floereqn} is equivalent to
\begin{align}
\label{floereqn2}
\overline{\partial}\doubletilde{u}+A\doubletilde{u}+\varphi(s)\nabla F_t(\doubletilde{u})=0,\qquad\doubletilde{u}(s,t+T)=\doubletilde{u}(s,t).
\end{align}
Note that the flow $\phi^A_t$ preserves Hilbert scales so that a solution to \eqref{floereqn} indeed gives us a solution to \eqref{floereqn2} of the same regularity. Note as well that the asymptotics $\lim_{s\to\pm\infty}\doubletilde{u}(s,t)$ of the solution $\doubletilde{u}$ to \eqref{floereqn2} are $T$-periodic solutions to \eqref{hamde}. \\\par

A result similar to our main theorem is proven in \cite{fabertnlse} for the nonlinear Schrödinger equation on projective Hilbert space (see also \Cref{nls1}). Because of the extra topology on projective Hilbert space, the author can prove the existence of infinitely many solutions rather than just one. \emph{We stress that our paper is self-contained, as in contrast to \cite{fabertnlse} we study the case of general Hamiltonian PDEs}.\\\par
%In particular, when $s>20\frac{1}{2}$ we find a strong solution to the nonlinear wave equation and when $s>26\frac{1}{2}$ we find a strong solution to the nonlinear Schrödinger equation, for generic time period $T$.
%In order to get a forced time-periodic solution for generic time period $T$, the nonlinearity $F_t$ shoulds be $s$-regularizing with $s>8\frac{1}{2}+4d$. In particular, we obtain strong $T$-periodic solutions to the nonlinear wave equation when $s>20\frac{1}{2}$ and for the nonlinear Schrödinger equation when $s>26\frac{1}{2}$ (see \Cref{nlwnls3}).
%\end{corollary}
%Details about the minimal requirements for NLS and NLW are in \Cref{nlwnls3}.\par
Let $F_t$ be any $A$-admissible nonlinearity with finite-dimensional restrictions $F^k_t:\C^{2k+1}\to\R$ given by $F^k_t(u):=F_t(u^k)$ with $u^k$ denoting the projection of $u\in\Hil$ onto the finite-dimensional subspace $\C^{2k+1}$. In analogy, for $G_t:=F_t\circ\phi_t^A$ let $G_t^k$ be its finite-dimensional restriction given by $G^k_t:=F^k_t\circ\phi_t^A$ with symplectic gradient $X_t^{G,k}$. In order to prove the main theorem for the infinite-dimensional nonlinearity $F_t$, we show that, after passing to a subsequence, finite-dimensional Floer curves $\widetilde{u}^k$ for the restricted nonlinearity $F_t^k$ converge as $k\to\infty$ to a Floer curve on the infinite-dimensional Hilbert space, as in the main theorem. This can be done because even though the time-$T$ free flow map is asymptotically degenerate, as our assumptions on the nonlinearity assure that this is no problem. \\\par 

Here $\widetilde{u}^k$ satisfies the Floer equation 
\begin{align*}
%\label{taufloereqn}
\overline{\partial}\widetilde{u}^k+\varphi_k(s)\nabla G_t(\widetilde{u}^k)=0,\qquad\widetilde{u}^k(s,t+T)=\phi_{-T}^A\widetilde{u}^k(s,t)
\end{align*}
with $\varphi_k(s)$ for $k\geq1$ meeting the requirements 
\[
\varphi_k(s)=0\;\mathrm{for}\;s\leq-1\;\mathrm{and}\;s\geq2k+1&;\qquad&&\varphi_k(s)=1\;\mathrm{for}\;s\in[0,2k];\\
0\leq\varphi'_k(s)\leq2\;\mathrm{for}\;s<0&;&&-2\leq\varphi'_k(s)\leq0\;\mathrm{for}\;s>0,
\]
such that $\varphi_k(s)\to\varphi(s)$ as $k\to\infty$ for every $s\in\R$. Furthermore, we impose the asymptotic condition $\lim_{s\to\pm\infty} \widetilde{u}^k(s,t)=0$.\\\par

This said, the main ingredient for the proof of \Cref{mainthm} is the following infinite-dimensional generalization of the Gromov-Floer compactness theorem, see \Cref{finallemma}. 

\begin{theorem}
There is a subsequence of the sequence $(\widetilde{u}^k)_k$ of Floer curves $\widetilde{u}^k:\R\times\R\to\C^{2k+1}$ which $C^{\floor{h/d}-1}_{\mathrm{loc}}$-converges to a solution $\widetilde{u}:\R\times\R\to\Hil$ of the Floer equation 
\[
(\partial_s+i\partial_t)\widetilde{u}+\varphi(s)\nabla G_t(\widetilde{u})=0,\qquad \wt{u}(s,t+T)=\phi^A_{-T}\wt{u}(s,t)
\]
as in \Cref{mainthm}.
\end{theorem}
After establishing the existence of a Floer curve, we can directly deduce the existence of a periodic orbit:
\begin{theorem}
Using finiteness of energy, the limit Floer curve $\widetilde{u}:\R\times \R\to\Hil$ satisfies the following asymptotic conditions: there exists sequences $s_n^\pm\in\R$ with $s_n^\pm\to\pm\infty$ as $n\to\infty$ such that
\begin{align*}
\lim_{n\to\infty}\widetilde{u}(s_n^-,t)=u_0=0,\qquad\lim_{n\to\infty}\widetilde{u}(s_n^+,t)=u_1(t),
\end{align*}
in the $C^{\floor{h/d}-1}$-sense. Here $u_0=0$ is the trivial and only fixed point of the free flow and $u_1$ is a $\phi_T^A$-periodic orbit of $G_t$.
\end{theorem}
We finish by discussing the regularity of the solution.
\begin{theorem}
The Floer curve $\wt{u}$, and in particular the $T$-periodic solution $u(t)=\phi_{t}^Au_1(t)$ we obtain from the $\phi_T^A$-periodic solution $u_1(t)$, is of regularity $h-d(r-1)-\frac{1}{2}$ for any $h> dr$, i.e. $\wt{u}:\R\times\R\to\Hil_{h-d(r-1)-1/2}\subset\Hil$.
\end{theorem}
In particular, when $h=\infty$ we obtain a smooth solution to the Floer equation and associated Hamiltonian PDE.\\\par
These results will play an important role for the construction of a symplectic cohomology theory for Hamiltonian PDEs with regularizing nonlinearities, which is an ongoing project of the authors. In an upcoming paper the authors will prove a Lagrangian version of the results above.

\section{Finite-dimensional case} 
\label{fdsection}
As mentioned above, the approach that we take is to start with the case of \emph{finite-dimensional nonlinearities}, that is, we consider nonlinearities which are given by the composition of any smooth $T$-periodic time-dependent map $F_t:\C^{2k+1}\to\R$ with bounded support with the orthogonal projection from $\Hil$ onto the finite-dimensional subspace $$\C^{2k+1}=\Span_\C\{z_{n}\}_{n=-k}^k\subset\Hil.$$ Note that any nonlinearity of this form is automatically $A$-admissible for any admissible $A$ and any admissible periods $(T,X)$. Since the linear symplectomorphism $\phi^A_t$ restricts to any $\C^{2k+1}$, it turns out that, in order to prove \Cref{mainthm} for these finite-dimensional nonlinearities, it suffices to replace the infinite-dimensional symplectic Hilbert space $\Hil$ by the finite-dimensional symplectic space $\C^{2k+1}$ and employ well established results of Floer theory in finite dimensions. In order to prove \Cref{mainthm} for general infinite-dimensional $A$-admissible nonlinearities, we will prove in the upcoming sections a generalized Gromov-Floer compactness result for the Floer curves introduced in this section. More precisely, we will consider the case when the dimension $k$ is allowed to vary, in particular, allowed to approach infinity. \\\par% Since the Hamiltonian $H_A$ is only densely defined but its flow $\phi_t^A$ is an everywhere defined symplectomorphism, we use Floer theory for symplectomorphisms as in \cite{salamondostoglou}. 

Let $F_t:\C^{2k+1}\to\R$ be any smooth $T$-periodic time-dependent map with bounded support in the ball $B_{R_k}(0)$ of radius $R_k>0$ and define again $G_t:=F_t\circ\phi_t^A$. Consider now the $\tau$-dependent Floer equation in $\C^{2k+1}$ 
\begin{align}
\label{taufloereqn}
\overline{\partial}\widetilde{u}+\varphi_\tau(s)\nabla G_t(\widetilde{u})=0,\qquad\widetilde{u}(s,t+T)=\phi_{-T}^A\widetilde{u}(s,t)
\end{align}
using the family of cut-off functions $\varphi_\tau:\R\to [0,1]$, $\tau\geq 0$ with $\varphi_0(s)=0$ and $\varphi_\tau(s)$ for $\tau\geq1$ meeting the requirements
\[
\varphi_\tau(s)=0\;\mathrm{for}\;s\leq-1\;\mathrm{and}\;s\geq2\tau+1&;\qquad&&\varphi_\tau(s)=1\;\mathrm{for}\;s\in[0,2\tau];\\
0\leq\varphi'_\tau(s)\leq2\;\mathrm{for}\;s<0&;&&-2\leq\varphi'_\tau(s)\leq0\;\mathrm{for}\;s>0.
\]
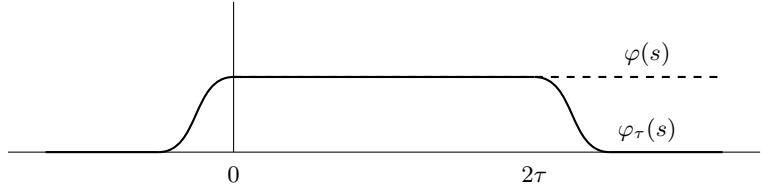
\begin{center}
\begin{tikzpicture}
\draw (0,0)--(10,0);
\draw (3,0)--(3,2);
\draw [thick] (2,0) to[out=0,in=180] (3,1); 
\draw [thick](3,1)--(7,1);
\draw [thick] (7,1) to[out=0, in=180] (8,0);
\draw [thick] (0.5,0)--(2,0);
\draw [thick,dashed] (3,1)--(9.5,1);
\draw [thick] (8,0)--(9.5,0);
\draw (3,-.3) node {$0$};
\draw (7,-.3) node {$2\tau$};
\draw (8.5,1.3) node {$\varphi(s)$};
\draw (8.5,.3) node {$\varphi_\tau(s)$};
\end{tikzpicture}
\end{center}
Our results stem from a careful analysis of the moduli space of curves satisfying this Floer equation \eqref{taufloereqn}. We define the moduli space for the finite-dimensional problem by
\[
\M^k:=\left\{\widetilde{u}^\tau:=(\widetilde{u},\tau)\in C^\infty(\R\times\R,\C^{2k+1})\times\R_{\geq0}\left\vert\right.\widetilde{u}\;\text{satisfies}\;\eqref{taufloereqn}\;\text{and}\;%\lim_{s\to\pm\infty}\widetilde{u}(s,t)\;\text{solves}\;\dot{u}=\varphi_\tau(s)X^{G,k}_t(u)
%\left\|\partial_s\widetilde{u}\right\|_2^2<\infty\right\}.
\lim_{s\to\pm\infty}\widetilde{u}(s,t)=0\right\}.
\]
After restricting to $\R\times[0,T]$, pictorially such Floer curves look like
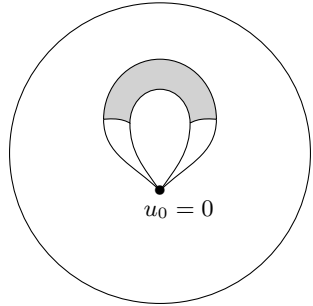
\begin{center}
\begin{tikzpicture}
\draw (0,0.5) circle (2cm);
\draw (0,0) node {$\bullet$};
\draw (.25,-.25) node {$u_0=0$};
\draw (0,0) to[out=45,in=270] (.75,.95);
\draw (0,0) to[out=135,in=270] (-.75,.95);

\draw (0,0) to[out=55,in=270] (.4,.9);
\draw (0,0) to[out=125,in=270] (-.4,.9);

\fill[fill=gray!35!white] (.75,.95) to[out=90,in=0] (0,1.75) to[out=180,in=90] (-.75,.95) plot [smooth,tension=1] coordinates {(-.75,.95) (-.55,.95) (-.4,.9)} to[out=90,in=180] (0,1.35) to[out=0,in=90] (.4,.9) to[out=15,in=185] (.55,.95) to (.75,.95);
\draw (.75,.95) to[out=90,in=0] (0,1.75) to[out=180,in=90] (-.75,.95) plot [smooth,tension=1] coordinates {(-.75,.95) (-0.55,.95) (-.4,.9)} to[out=90,in=180] (0,1.35) to[out=0,in=90] (.4,.9) plot [smooth,tension=1] coordinates {(.4,.9) (.55,.95) (.75,.95)};

%\fill[fill=green!20!white] (.75,.95) to[out=90,in=0] (0,1.75) to[out=180,in=90] (-.75,.95) to [smooth,tension=1] (-.4,.9) to[out=90,in=180] (0,1.35) to[out=0,in=90] (.4,.9) to [smooth,tension=1] (.75,.95);
%\draw (.75,.95) to[out=90,in=0] (0,1.75) to[out=180,in=90] (-.75,.95) to [smooth,tension=1] (-.4,.9) to[out=90,in=180] (0,1.35) to[out=0,in=90] (.4,.9) to [smooth,tension=1] (.75,.95);

%\draw (.75,.95) to[out=90,in=0] (0,1.75);
%\draw (-.75,.95) to[out=90,in=180] (0,1.75);

%\draw [name path=one] plot [smooth,tension=1] coordinates {(0,0) (-0.75,.95) (0,1.75) (.75,.95) (0,0)};
%\draw [name path=two] plot [smooth,tension=1] coordinates {(0,0) (-0.35,.75) (0,1.25) (.35,.75) (0,0)};
%\draw [name path=three] plot [smooth,tension=1] coordinates {(-.35,.90) (-.55,.95) (-.75,.95)};
%\draw [name path=four] plot [smooth,tension=1] coordinates {(.35,.90) (.55,.95) (.75,.95)};
%\tikzfillbetween[
%    of=one and two and three and four,split
%  ] {pattern=north west lines};
\end{tikzpicture}
\end{center}
where the gray area depicts the part where $\varphi_\tau(s)=1$. In order to show that we can compactify $\M^k$, we crucially use the bounded support condition in \Cref{admissible} and the following result.
\begin{proposition}[Maximum principle]
\label{maxprinc}
If $(\Sigma,j)$ is a Riemann surface and $\widetilde{u}:(\Sigma,j)\to(\C^{2k+1},i)$ a holomorphic map, then 
\[
\Sigma\to[0,\infty):z\mapsto \vert \widetilde{u}(z)\vert^2
\]
has no local maximum.
\end{proposition}
This implies that Floer curves $\widetilde{u}$ cannot escape the ball $B_{R_k}(0)$: if they would, they would be holomorphic outside the ball, where $G_t=0$, and so by the above they could not have a maximum which is impossible. So even though the target space of the Floer curve is not compact, the image is contained in a compact set. We have the following
\begin{proposition}
\label{fdresult}
For every $\tau\in\N$ there is a Floer curve $\widetilde{u}^{\tau}$ in $\M^k$. %After passing to a subsequence, they converge locally with all derivatives to a Floer curve $\widetilde{u}$ as stated in the main theorem. In particular, there exists a forced periodic solution $u_1$.   
\end{proposition}
\begin{proof}
For the proof we use well-known results from Floer theory such as written in e.g. \cite{salamon}, and \cite{salamondostoglou} for Floer theory for general symplectomorphisms; since all these results are well-established in the literature, we freely use established terminology without giving definitions. Note that since $H_A$ is smooth on finite-dimensional subspaces of $\Hil$, one can either use a solution $\wt{u}$ to \eqref{taufloereqn} or $\doubletilde{u}$ solving
\[
\overline{\partial}\doubletilde{u}+A\doubletilde{u}+\varphi_\tau(s)\nabla F_t(\doubletilde{u})=0,\qquad \doubletilde{u}(s,t+T)=\doubletilde{u}(s,t).
\]
For the start note that the energy $E(\widetilde{u})=\left\|\partial_s\widetilde{u}\right\|_{L^2}^2$ of the Floer curves is uniformly bounded by $4T\| F\|_{C^0}$ (see \cite[chapter 8]{mcduffjholo}) which is finite by \Cref{admissible} condition $3$.
%As mentioned in \Cref{floerthy}, this latter condition is equivalent to the asymptotic condition that $u_\pm(t):=\lim_{s\to\pm\infty}\widetilde{u}(s,t)$ satisfies 
%\[
%u_\pm(t)=X^{G,\pm}_t(u_\pm(t))
%\]
%where $X^{G,\pm}_t$ is the symplectic gradient of $\lim_{s\to\pm\infty}\varphi_\tau(s)G_t$. If $\tau\in\R$ then $X^{G,\pm}_t=0$ of course, so that $u_\pm(t)\equiv0$. 
Assuming transversality for the nonlinear Cauchy-Riemann operator for the moment, the moduli space of such pairs $(\widetilde{u},\tau)$ is a $1$-dimensional manifold. Since for $\tau=0$ the unique Floer curve $(\widetilde{u},0)$ is the constant curve $\widetilde{u}\equiv0$, the moduli space is not empty. Indeed Floer curves $(\widetilde{u},\tau)$ exist for all $\tau>0$ by Gromov-Floer compactness, as we can exclude bubbling-off of holomorphic spheres as well as breaking-off of cylinders for finite $\tau$. Note that existence of holomorphic spheres is excluded due to the fact that the symplectic form is exact.  %Hence the moduli space $\M^k$ can be compactified by adding broken holomorphic curves. By this we mean the following: a sequence of Floer curve $\widetilde{u}^{\tau}$ \textit{breaks} if the limit consists of two concatenated Floer curves, one connecting the trivial solution $u_0\equiv0$ to a solution $u_1$ of the nonlinear PDE, and one connecting $u_1$ to $u_0$. We then have found a Floer curve which satisfies the asymptotic condition \eqref{asymptoticcondition} in $\C^{2k+1}$. 
%\begin{center}
%\begin{tikzpicture}
%\draw (0,.5) circle (2cm);
%\draw (0,0) node {$\bullet$};
%\draw (.25,-.25) node {$u_0$};
%\draw (0,0) to[out=45,in=270] (.75,.95);
%\draw (0,0) to[out=135,in=270] (-.75,.95);
%
%\draw (0,0) to[out=55,in=270] (.4,.9);
%\draw (0,0) to[out=125,in=270] (-.4,.9);
%
%\
%
%
%\fill[fill=gray!35!white] (.75,.95) to[out=90,in=0] (0,1.75) to[out=180,in=90] (-.75,.95) plot [smooth,tension=1] coordinates {(-.75,.95) (-.55,.95) (-.4,.9)} to[out=90,in=180] (0,1.35) to[out=0,in=90] (.4,.9) to[out=15,in=185] (.55,.95) to (.75,.95);
%\draw (.75,.95) to[out=90,in=0] (0,1.75) to[out=180,in=90] (-.75,.95) plot [smooth,tension=1] coordinates {(-.75,.95) (-0.55,.95) (-.4,.9)} to[out=90,in=180] (0,1.35) to[out=0,in=90] (.4,.9) plot [smooth,tension=1] coordinates {(.4,.9) (.55,.95) (.75,.95)};
%
%\draw [thick] (0,1.75) to[out=-70,in=70] (0,1.35);
%\draw (0,1.75) [thick, densely dashed] to[out=-110,in=110] (0,1.35);
%
%\draw (-1,.95) node {$\widetilde{u}$};
%\draw (0,1.95) node {$u_1$};
%
%\draw [thick] (0,0) to[out=135,in=270] (-.75,.95) to[out=90,in=180] (0,1.75);
%\draw [thick] (0,0) to[out=125,in=270] (-.4,.9) to[out=90,in=180] (0,1.35);
%\end{tikzpicture}
%\end{center}
%In particular, $u_1$ is a (non-trivial) solution to the finite-dimensional Hamiltonian equation $\dot{u}_1=X^{G,k}_t(u_1)$.\par
Note that the assumption that the Hamiltonian PDE with $F_t=0$ only has the trivial periodic solution $u_0=0$ is essential here to conclude that breaking of Floer curves cannot happen for finite $\tau>0$. \par
%the solution $u_1$ at which the curve breaks is a solution to the nonlinear PDE: when the curve would break where $\varphi$ is zero, then one component would be trivial by energy reasons. \par%Hence we may assume that it breaks where $\varphi=1$, which, in particular, implies that the breaking can only occur when $\tau\to\infty$. Finally the solution $u_1$ is indeed nontrivial as it does not lie outside the support of $F_t$, since else it would be a nontrivial fixed point of the free flow, which does not exist. Alternatively, the last claim also directly follows from the maximum principle.\par
It remains to discuss the problem with transversality of the perturbed Cauchy-Riemann operator $\overline{\partial}+\varphi_\tau(s)\nabla G_t^k$. Since we cannot expect transversality to hold, we first need to approximate $i$ by a family of time-dependent almost-complex structures $J_t^\nu$ satisfying $(\phi_{-T}^A)_*J_t^\nu=J_{t+T}^\nu$, in the sense that $J_t^\nu\to J^0_t=i$ as $\nu\to\infty$. We assume that the perturbed almost complex structure $J_t^\nu$ agrees with $i$ outside the ball $B_{R_k}(0)$ so that the maximum principle still holds. The existence of Floer curves as claimed above then holds for all $\nu\neq0$ and by applying Gromov-Floer compactness as $\nu\to 0$ this implies the existence of Floer curves for $\nu=0$. % Note, in particular, that we have not only proven the existence of a broken Floer curve and a nontrivial solution of the nonlinear PDE but also the existence of Floer curves $\widetilde{u}^\tau$ for all $\tau>0$.
\end{proof}

\section{Bubbling-off analysis}
\label{bubblingoffsection}
%The proof of the main theorem \Cref{mainthm} relies on the fact that, with our assumptions, the non-linearity can be approximated sufficiently well by finite-dimensional ones for the existence of a Floer curve to still be guaranteed. Note that we cannot expect this to hold for any type of non-linearity, as the time-$T$ flow of the vector field $G_t$ could very well have a ((sub)sequence of) eigenvalue(s) equal to (or converging to) $1$. 
%We start our series of convergence and boundedness lemmas with the following lemma. 
After settling the case of finite-dimensional nonlinearities in the previous section, we start by recalling the detailed strategy for the case of general infinite-dimensional $A$-admissible nonlinearities. Let $F_t$ be any $A$-admissible nonlinearity with finite-dimensional restrictions $F^k_t:\C^{2k+1}\to\R$ given by $F^k_t(u):=F_t(u^k)$ with $u^k$ denoting the projection of $u\in\Hil$ onto the finite-dimensional subspace $\C^{2k+1}$. In analogy, for $G_t:=F_t\circ\phi_t^A$ let $G_t^k$ be its finite-dimensional restriction given by $G^k_t:=F^k_t\circ\phi_t^A$ with symplectic gradient $X_t^{G,k}$. In order to prove the main theorem for the infinite-dimensional nonlinearity $F_t$, we choose for every $k\in\N$ a Floer curve $\wt{u}^k$ for the restricted nonlinearity $F_t^k$ such that $(\wt{u}^k,k)\in\M^k$. We then show that, after passing to a subsequence, these finite-dimensional Floer curves converge as $k\to\infty$ to a Floer curve on the infinite-dimensional Hilbert space, as in the main theorem. This can be done because even though the time-$T$ free flow map is asymptotically degenerate, as our assumptions on the nonlinearity assure that this is no problem. Note that $\widetilde{u}^k$ satisfies a $\tau$-dependent Floer equation with $\tau=k$ and $\varphi_k(s)\to\varphi(s)$ as $k\to\infty$ for every $s\in\R$. \\\par 

As a first step we would like to bound the Floer curves $\widetilde{u}^k$, for all $k$, in the $C^m$-norm, where $m=\floor{h/d}$. We will do this by showing the first derivatives are bounded and then using an elliptic bootstrapping argument. We use ideas similar to those in \cite{fabertnlse}. We stress, however, that contrary to \cite{fabertnlse} for our problem we work on linear space and with general Hamiltonians with minimal regularity. \\\par 

We start by proving the analogue of \Cref{lemma1} about the convergence of $G^k:\R\times\Hil\to\R$, $G^k(t,u)=G^k_t(u)$ to $G:\R\times\Hil\to\R$, $G(t,u)=G_t(u)$. Note that we explicitly want to include into our discussion not only the derivatives with respect to $u\in\Hil$, but also the derivatives with respect to the time $t\in\R$. 
%Let us write $\widetilde{u}^k$ for a $k$-dimensional solution to \eqref{taufloereqn} with $\tau=k$, in which case we omit the $\tau$-dependence in the notation.

\begin{lemma}
\label{lemma4}
$\nabla G^k:\R\times\Hil\to\Hil$ converges to $\nabla G:\R\times\Hil\to\Hil$ uniformly with all $u$-derivatives, and with all $t$-derivatives up to order $m=\floor{h/d}$ (which is at least two). Furthermore, the Fourier coefficients of $\nabla G$ in the expansion
\[
\nabla G(u)(t)=\sum_{p,n}\widehat{\nabla G(u)}(p,n)e^{i2\pi pt/T}z_n
\]
with respect to $n\in\Z$, $p\in\Z-an^dT/(2\pi)$ satisfy 
%$\left\|G-G^k\right\|_{C^0},
\[
\widehat{\nabla G(u)}(p,n)|n|^h |p|^m\to0\text{ as }|n|,|p|\to\infty.
\]

%, where in the first case the supremum norm is taken over the set $U_R^k\subset\Hil$ of all points in $\Hil$ which have distance less than $R$ from $\C^k$ for any fixed $R>0$ .
\end{lemma}
\begin{proof} 
The statement about the $u$-derivatives directly follows from \Cref{lemma1}, as the $u$-derivatives of $G^k$ are obtained from the $u$-derivatives of $F^k$ by composition with the linear unitary map $\phi^A_t$. In order to compute the $t$-derivatives of $G$ one does not only have to take the $t$-derivatives of $F$ into account, but also the $t$-derivatives of $\phi^A$. While $F$ is assumed to be smooth in $t\in\R$, the $t$-derivatives of $\phi^A$ are given by $\partial_t^{\alpha}\phi^A=(JA)^{\alpha}\cdot\phi^A:\Hil_0\to\Hil_{-\alpha d}$ and hence have decreasing regularity. But since $F$ is assumed to be $h$-regularizing and hence $F_t$ extends to a smooth map $\Hil_{-h}\to\R$ with $h> dr$, it follows that derivatives up to order $m=\floor{h/d}\geq2$ are no problem. Moreover, since, in analogy with \Cref{lemma1}, we already know that $\nabla F^k$ converges to $\nabla F$ uniformly with all derivatives when viewed as maps from $\Hil_{-h}$ into $\Hil_{h}\subset\Hil$, it follows that $\partial_t^{\alpha} G^k$ converges to $\partial_t^{\alpha} G$ for $k\to\infty$ as long as $\alpha\leq m$. When we expand
\[
\nabla G(u)(t)=\sum_{p,n}\widehat{\nabla G(u)}(p,n)e^{i2\pi pt/T}z_n
\]
the statement as $|p|\to\infty$ follows. The statement for the $n$-variable follows from the fact that $\nabla G(u)$ is in $\Hil_h$. 
%At the end of \Cref{sdp} we show that we can without loss of generality take the $C^0$ of $G-G^k$ over $U_R^k$ instead of over $\Hil$ by showing that the Floer curves can never leave this region.
\end{proof}

We continue with a lemma about the first derivatives.
\begin{lemma}
\label{firstderiv}
The first derivatives of the Floer curves $\widetilde{u}^k$ are bounded uniformly in $k$, i.e. $\sup_{k}\left\| T\widetilde{u}^k\right\|_{C^0}<\infty$. 
\end{lemma}
\begin{proof}
Showing that the first derivatives are bounded is done by assuming that 
\begin{align}
\label{assumption1}
\sup_k\left\| T\widetilde{u}^k\right\|_{C^0}=\infty
\end{align}
and showing that this assumption leads to the formation of a sphere. We will not argue, as in the finite-dimensional case, that because $\omega$ is exact no holomorphic spheres can exist: this would require Gromov-Floer compactness in infinite dimensions. Rather, assuming the first derivative is unbounded, we show that a sphere is being formed as image of the disc where the length of the image of the boundary of the disc converges to zero. We then bound the derivative of the Floer curve by the symplectic area of these discs, which by exactness of $\omega$ is given by an integral over the boundary, thereby deriving a contradiction. This implies boundedness in the $C^1$-norm. Although the proof is very similar to the proof of the well-established finite-dimensional result, we include it with all details as our infinite-dimensional result does \emph{not} follow from the finite-dimensional bubbling-off result. \\\par 

Hence assume that the first derivative is unbounded in the sense that for
\[
C_k:=\max_{z=(s,t)\in\R\times\R}\left\{|\partial_s \widetilde{u}^k(z)|\right\}=:\left|\partial_s\widetilde{u}^k(z_k)\right|
\]
the sequence $(C_k)_{k\in I}$ converges to $\infty$ for some index-set $I$. We can assume that the Floer curve $\widetilde{u}^k$ attains this maximum at some point $z_k$ because of the asymptotic conditions. Now we reparametrize
\[
\widetilde{v}^k:B_{\sqrt{C_k}}(0)\to\C^{2k+1}:z\mapsto \widetilde{u}^k\left(\frac{z}{C_k}+z_k\right)
\]
so that $|\partial_s\widetilde{v}^k(0)|=1$ and $|\partial_s\widetilde{v}^k(z)|\leq1$ for $|z|\leq\sqrt{C_k}$. Then we define a family of maps $\gamma_r^k$ for $0\leq r\leq\sqrt{C_k}$ by
\[
\gamma_r^k:S^1\to\C^{2k+1}:\theta\mapsto \widetilde{v}^k(re^{i\theta}).
\]
Let $L:C^\infty(S^1,\C^{2k+1})\to\R$ be the map which assigns to a loop its length with respect to the metric $\omega(\cdot,i\cdot)$ restricted to $\C^{2k+1}$. Let $A:C^\infty(B_R(0),\C^{2k+1})\to\R$ be the area functional $A(v):=\int v^*\omega$, where again we restrict the symplectic form $\omega$ to $\C^{2k+1}$. %Now by exactness
%\[
%A(\widetilde{v}^k)=L(\gamma_{\sqrt{C_k}}^k).
%\]
Now we show that for increasing dimension $k$, the length of the image of the boundary circle decreases. More precisely, we show that for all $k$, there exists $\frac{\sqrt{C_k}}{2}\leq r_k\leq\sqrt{C_k}$ such that $L(\gamma_{r_k})\to0$. By the exactness of $\omega$ the area of $\widetilde{v}^k_{r_k}$, which is the restriction of $\widetilde{v}^k$ to the disk of radius $r_k$, goes to zero.\par
%for $\widetilde{v}^k_{r_k}:D^2_{r_k}(0)\to\C^{2k+1}$. 
As a first step, we show that $A$ is bounded by the energy of the solution $\widetilde{u}^k$ as $k\to\infty$, which will show that the area is bounded.
\[
A(\widetilde{v}^k)&=\int_{B_{\sqrt{C_k}}(0)} \widetilde{v}^{k*}\omega\\
&=\int_{B_{\sqrt{C_k}}(0)}\omega(\partial_s\widetilde{v}^k,\partial_t\widetilde{v}^k)ds\wedge dt\\
%%&=\int_{D^2_{\frac{1}{\sqrt{C_k}}(z^k)}}\omega(\partial_s\widetilde{u}^k,\partial_t\widetilde{u}^k)ds\wedge dt\\
%%&=\int_{D^2_{\frac{1}{\sqrt{C_k}}(z^k)}}\left|\partial_s\widetilde{u}^k\right|^2+\omega(\partial_s\widetilde{u}^k,\varphi_k(s)X_t^{G,k}(\widetilde{u}^k))ds\wedge dt\\
%&\leq E(\widetilde{u}^k)+\frac{1}{2}\int_{D^2_{\frac{1}{\sqrt{C_k}}(z^k)}}\left|\varphi_k(s)X_t^{G,k}(\widetilde{u}^k)\right|^2ds\wedge dt\\
&\leq E(\widetilde{u}^k)+\int_{B_{\frac{1}{\sqrt{C_k}}(z^k)}}\varphi_k(s)dG_t^k(\partial_s\widetilde{u}^k)ds\wedge dt\\
&\leq E(\widetilde{u}^k)+\int_{B_{\frac{1}{\sqrt{C_k}}(z^k)}}\left\|G^k\right\|_{C^1}ds\wedge dt
\]
Since $\sqrt{C_k}\to\infty$ by our assumption \eqref{assumption1}, the second term vanishes. Now we write $\widetilde{v}^k(z)=\widetilde{v}^k(re^{i\theta})$ and, assuming $k$ is sufficiently large, compute
\[
\int_{\sqrt{C_k}/2}^{\sqrt{C_k}} rL\left(\gamma_{r}^k\right)^2dr&=\int_{\sqrt{C_k}/2}^{\sqrt{C_k}}r\left(\int_0^{2\pi}\left|\partial_\theta\widetilde{v}^k(re^{i\theta})\right|d\theta\right)^2dr\\
&\leq2\pi \int_{\sqrt{C_k}/2}^{\sqrt{C_k}}\int_0^{2\pi}r\left|\partial_\theta\widetilde{v}^k\right|^2d\theta dr\\
&\leq10\pi T\|F\|_{C^0}
\]
using Cauchy-Schwarz, the previous inequality and the fact that $E(\widetilde{u}^k)< 5T\|F\|_{C^0}$ (see \cite{mcduffjholo}). By setting $L_0^k$ to be the minimum of $L(\gamma^k_{r})$ for $\sqrt{C_k}/2\leq r\leq\sqrt{C_k}$, we get
\[
10\pi T\|F\|_{C^0}%(\widetilde{u}^k)
&\geq \int_{\sqrt{C_k}/2}^{\sqrt{C_k}}r(L_0^k)^2dr\\
&=\frac{3(L_0^k)^2C_k}{8}
\]
so that 
\[
L_0^k\leq\sqrt{\frac{80\pi T\|F\|_{C^0}}{3C_k}}
\]
which tends to zero as $k\to\infty$. Since $\omega=d\lambda$, for any disc $v:B_R(0)\to\C^{2k+1}$ we have
\[
A(v)=\int_{B_R(0)}v^*\omega=\int_{\partial B_R(0)}v^*\lambda
\]
and so the area $A(\widetilde{v}^k_{r_k})\to0$ as $L_0^k\to0$. Now there are two ways to prove the desired result. First, it follows from the a priori estimate 
\[
\left|\partial_s\widetilde{v}^k(0)\right|^2<c\frac{A(\widetilde{v}^k_{r_k})}{r_k^2}
\]
in \cite[chapter 4]{mcduffjholo} by observing that the Floer curve can be realized as an actual $\phi_T^{H,k}$-periodic $J$-holomorphic curve when we set $J^k_t:=(\phi_{-t}^{G,k})_*i$. Note that contrary to \cite[chapter 4]{mcduffjholo} we don't work with a single almost complex structure $J$ but with a sequence $J_t^k$ which converges to $J_t=(\phi_{-t}^G)_*i$. Since we have $|\partial_s\widetilde{v}^k(0)|=1$, the contradiction then follows by letting $r_k\to\infty$. \par
Alternatively, consider the following.
%
%Therefore we can employ the a priori estimate  %(\textbf{is this a correct reference?})
%\[
%\left|\partial_s\widetilde{v}^k(0)\right|^2<c\frac{A(\widetilde{v}^k_{r_k})}{r_k^2}
%\]
%for some $c>0$, which we derive below, gives us the contradiction we are looking for.\par%: for $k\to\infty$ the symplectic area $A(\widetilde{v}^k_{r_k})$ is bounded uniformly in $k$ by $10\pi T\|F\|_{C^0}$, while $r_k\to\infty$ so that the derivative at $0$ is zero, which we assumed it was not. \par
%There are two ways to see that this a priori estimate holds: first it follows from the a priori estimate in \cite[chapter 4]{mcduffjholo} by observing that the Floer curve can be realized as an actual $\phi_T^{H,k}$-periodic $J$-holomorphic curve when we set $J^k_t:=(\phi_{-t}^{G,k})_*i$. Note that contrary to \cite[chapter 4]{mcduffjholo} we don't work with a single almost complex structure $J$ but with a sequence $J_t^k$ which converges to $J_t=(\phi_{-t}^G)_*i$. \par
%%we can also prove the estimate 
%%
%%directly: we first note that
We first observe that
\[
\overline{\partial}\widetilde{v}^k=-C_k^{-1}\varphi_k(s)\nabla G_t^k(\widetilde{v}^k)\to0\quad\text{as}\quad k\to\infty
\]
and so
\[
\Delta\partial_s\widetilde{v}^k=-(\partial\circ\partial_s)C_k^{-1}\varphi_k(s)\nabla G_t^k(\widetilde{v}^k)\to0\quad\text{as}\quad k\to\infty.
\]
Writing $v:=\partial_s\widetilde{v}^k$ and using the divergence theorem, we get
\[
\partial_\rho\left(\frac{1}{\rho}\int_{\partial B_\rho(0)}v\right)=\frac{1}{\rho}\int_{B_\rho(0)}\Delta v\to0\quad\text{as}\quad k\to\infty
\]
uniformly in $\rho$ for $\rho\leq\epsilon$ for some $\epsilon>0$. Using the fact that $(2\pi\rho)^{-1}\int_{\partial B_\rho(0)}v\to v(0)$ as $\rho\to0$ as well as the above convergence to $0$ as $k\to\infty$, we get
\[
v(0)-\frac{1}{\pi \epsilon^2}\int_{ B_\epsilon(0)}v(z)dz\to0\quad\text{as}\quad k\to\infty.
\]
Now
\[
\frac{1}{\pi \epsilon^2}\left|\int_{B_\epsilon(0)}v(z)dz\right|\leq\frac{1}{\pi^{1/2}\epsilon}\left(\int_{B_\epsilon(0)}|v(z)|^2dz\right)^{1/2}\leq\frac{1}{\pi^{1/2}\epsilon}\|v\|_{L^2}
\]
so that indeed
\[
\left|\partial_s\widetilde{v}^k(0)\right|^2<c\frac{A(\widetilde{v}^k_{\epsilon})}{\epsilon^2}
\]
for $k$ sufficiently large and some positive constant $c$ which is independent of the dimension. Since $A(\widetilde{v}^k_\epsilon)\to0$ as $k\to\infty$ we obtain a contradiction to the fact $|\partial_s\widetilde{v}^k(0)|=1$.
%Now we can derive a contradiction, showing that $\partial_s\widetilde{u}^k$ must be bounded for all $k$: there is a $2$-disk $\widetilde{D}^2_{R_k}\subset\C^k$ such that 
%\[
%S_k:=\widetilde{D}^2_{R_k}\cup\widetilde{v}^k(D^2_{r_k}(0))
%\]
%is a sphere in $\C^k$. By exactness of the symplectic form, we have $\int_{S_k}\omega=0$, but
%\[
%\int_{S_k}\omega&=\int_{\widetilde{D}^2_{R_k}}\omega+\int_{D^2_{r_k}(0)}\widetilde{v}^{k*}\omega\\
%&\leq cL_0+A(\widetilde{v}^k_{r_k})
%\]
%for some constant $c$. The length $L_0$ tends to zero as $k\to\infty$ and $A(\widetilde{v}^k_{r_k})$ is the symplectic area of the disk $D^2_{r_k}(0)$. Since $r_k\to\infty$, the area is non-zero. Therefore we conclude that the sequence $(C_k)_k$ does not tend to $\infty$, so that $\partial_s\widetilde{u}^k$ is bounded. Since
%\[
%\partial_s\widetilde{u}^k=\varphi_k(s)\nabla G_t^k(\widetilde{u}^k)-i\partial_t\widetilde{u}^k
%\]
%and $f(k)=\left\| G_t^k\right\|_{C^1}$ is bounded, also $\partial_t\widetilde{u}^k$ is bounded, so that $\sup_k\left\| T\widetilde{u}^k\right\|_\infty\leq\infty$. 
\end{proof}
We can now apply the aforementioned bootstrapping argument, to show boundedness of the Floer curves in the $C^m$-norm. Recall that $m=\floor{h/d}\geq2$.
\begin{proposition}
\label{bootstrapping}
The Floer curves $\wt{u}^k$ are $C^m$-bounded uniformly in $k$, that is
\[
\sup_k\left\|\widetilde{u}^k\right\|_{C^m}<\infty.
\]
\end{proposition}
\begin{proof}
For the proof we choose the bounded open subset $B=(s+\Delta s,s-\Delta s)\times(0,1)\subset\R^2$ for some fixed $\Delta s$ and all norms are understood after restricting the maps $\widetilde{u}^k$ to this bounded open subset. By the result above, and the discussion following the maximum princple, we know that $\left\|\widetilde{u}^k\right\|_{C^1}$ is bounded. We will use the fact that our sequence of finite-dimensional nonlinearities approximates the original one and an elliptic bootstrapping argument, to show boundedness in all $C^\alpha$-norms up to $\alpha=m$. By the Sobolev embedding theorem (see e.g. \cite{brezis}), the inequality 
\[
\left\|\widetilde{u}^k\right\|_{C^\beta}\leq c_0\left\|\widetilde{u}^k\right\|_{W^{\alpha,p}}
\]
holds for $p>2$, with $\alpha,\beta\in\N$ and for all $\beta\leq\alpha-2/p$, and a constant $c_0>0$ which is independent of the dimension of the codomain. 
It follows that it suffices to show boundedness of $\widetilde{u}^k$ in the $W^{\alpha,p}$-norms up to $\alpha=m+1$. \par
We first observe that the boundedness in $C^1$ implies boundedness in $W^{1,p}$; note that this is the point where it is crucial that we first restrict $\widetilde{u}^k$ to a bounded open subset. Assume now that $\left\|\widetilde{u}^k\right\|_{W^{\alpha,p}}$ is bounded for some $\alpha>1$, uniformly in $k$. We have that $\widetilde{u}^k$ satisfies
\[
\overline{\partial}\widetilde{u}^k=-\varphi_k(s)\nabla G_t^k(\widetilde{u}^k)=:\eta^k
\]
and $\eta^k$ is bounded in the $W^{\alpha,p}$-norm iff the $W^{\alpha,p}$-norm of $\nabla G_t^k(\widetilde{u}^k)$ is bounded with
\[
{\nabla G^k(\widetilde{u}^k)(s,t)=\nabla G^k_t(\widetilde{u}^k(s,t))}.
\]
On the other hand, viewing $\nabla G^k(\widetilde{u}^k):B\to\C^{2k+1}$ as a composition of the maps $\check{u}^k:B\to B\times\C^{2k+1}:(s,t)\mapsto (s,t,\widetilde{u}^k(s,t))$ and $\nabla G^k:B\times\C^{2k+1}\to\C^{2k+1}:(s,t,u)\mapsto \nabla G^k_t(u)$, by \cite[appendix B]{mcduffjholo} it holds true that 
\[
\left\|\nabla G^k(\widetilde{u}^k)\right\|_{W^{\alpha,p}}\leq c_1\left\|\nabla G^k\right\|_{C^{\alpha}}\left(\left\|\check{u}^k\right\|_{C^0}^{\alpha-1}+1\right)\left(\left\|\check{u}^k\right\|_{W^{\alpha,p}}+1\right)
%\left(\cdot \left\|\nabla G^k\right\|_{C^{\alpha}}+1\right)\left(\left\|\widetilde{u}^k\right\|_{W^{\alpha,p}}+1\right)
\]
with a constant $c_1>0$ which is independent of the dimension of the target space. %Here we use that $C^{\alpha}\subset C^{\alpha-1,1}$. %Note that here we view $\nabla G^k$ as a map from $B\times\C^{2k+1}$ to $\C^{2k+1}$, with $B\subset\R^2$ a bounded open subset, given by $\nabla G^k(s,t,u)=\nabla G^k_t(u)$; 
Note that the $C^{\alpha}$-norm of $\nabla G^k$ also contains $t$-derivatives of $t\mapsto \nabla G^k_t$. Since by \Cref{lemma4} we have for all $\alpha\leq m$ that $\|\nabla G^k\|_{C^{\alpha}}\to\|\nabla G\|_{C^{\alpha}}$ as $k\to\infty$, it follows that $\|\nabla G^k\|_{C^{\alpha}}$ is bounded for $\alpha\leq m$. Together with the induction hypothesis, we get boundedness of $\nabla G_t^k(\widetilde{u}^k)$ in the $W^{\alpha,p}$-norm as long as $\alpha\leq m$. Now local regularity of the Cauchy-Riemann operator $\overline{\partial}$ together with boundedness of $\eta$ in the $W^{\alpha,p}$-norm, implies
\[
\left\|\widetilde{u}^k\right\|_{W^{\alpha+1,p}}\leq c_2\left(\left\|\overline{\partial}\widetilde{u}^k\right\|_{W^{\alpha,p}}+\left\|\widetilde{u}^k\right\|_{L^p}\right)
\]
is finite for $\alpha\leq m$. Note that, again, $c_2>0$ is independent of the dimension of the codomain. Finally we remark that all constants depend on the bounded open subset $B$ but not on $s$, so that we obtain a bound which is uniform in $s$.
\end{proof}

%\newpage

\section{Small divisor problem}
\label{sdp}
We have chosen the setting such that the nonlinearity can be approximated by finite-dimensional ones better than the eigenvalues of the time-$T$ flow of the free Hamiltonian approach $1$. In this section, we will make this statement precise by giving bounds on the norms of the tail of $\widetilde{u}^k$, and invoke a result from number theory to overcome the small divisor problem which arises as we increase the dimension $k$.\par
Let us write a finite-dimensional solution of the Floer equation \eqref{taufloereqn} as
\[
\widetilde{u}^k=(\widetilde{u}^{k,\ell},\widetilde{u}^{k,\ell}_\perp):\R\times\R\to\C^{2\ell+1}\oplus\C^{2k-2\ell}=\C^{2k+1}\subset\Hil
\]
and call the tail $\widetilde{u}^{k,\ell}_\perp$ of $\widetilde{u}^k$ the \textit{normal component}. The desired statement (\Cref{normcompprop2}) needed for the proof in \Cref{completingproof} of the main theorem, is then that we have
\begin{align}
\label{normalcomp1}
\sup_{k\geq\ell}\left\|\widetilde{u}^{k,\ell}_\perp\right\|_{C^{m-1}}\to0\quad\mathrm{as}\quad \ell\to\infty
\end{align}
for $m=\floor{h/d}$. We prove this by observing that the Fourier coefficients of the Floer curve, which depend on the $s$-coordinate, satisfy an ODE involving the Fourier coefficients of the Hamiltonian vector field of the nonlinearity and which satisfy a decay property as $s\to\pm\infty$. The following elementary lemma then allows us to show that the coefficients themselves decay to zero with some rate which we compute.
\begin{lemma}
\label{elementarylemma}
Let $w=w_R+iw_I:\R\to\C$ be a continuously differentiable solution to the ODE with asymptotic condition
\begin{align}
\label{ode}
w'(s)=\lambda w(s)+f(s);\qquad w(s)\to0\text{ as }s\to\pm\infty
\end{align}
where $\lambda\in\R$. If $f=f_R+if_I:\R\to\C$ satisfies $\op{f}_{C^0}<\infty$, then $\op{w}_{C^0}\leq\sqrt{2}\op{f}_{C^0}/|\lambda|$.
\end{lemma}
\begin{proof}
The proof is by contradiction: assume $|w(s_0)|>\sqrt{2}\op{f}_{C^0}/|\lambda|$ for some $s_0\in\R$ and, without loss of generality, that $|w_R(s_0)|\geq|w_I(s_0)|$ so that $|w_R(s_0)|>\op{f}_{C^0}/|\lambda|$ by the Pythagorean theorem. Assume that $w_R(s_0)>0$ and $\lambda>0$ (different signs lead to obvious changes in the proof). Since $w(s)\to0$ as $s\to+\infty$, by the intermediate value theorem we know that there is some $s_1>s_0$ such that $w_R(s_1)=\op{f}_{C^0}/\lambda$ and $w_R(s)>\op{f}_{C^0}/\lambda$ for all $s\in(s_0,s_1)$. By the mean value theorem, there exists $s_2\in(s_0,s_1)$ such that $w_R'(s_2)<0$. Since $|f_R(s)|\leq|f(s)|\leq\op{f}_{C^0}$, we have $w_R'(s_2)<0$ but $\lambda w_R(s_2)+f_R(s_2)>0$, which contradicts the assumption that $w$ satisfies \eqref{ode}.
\end{proof}
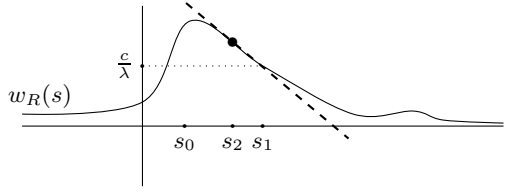
\begin{center}
\begin{tikzpicture}
\begin{scope}[scale=.8]
\draw (-2,0)--(6,0);
\draw (0,-1)--(0,2);
\draw (.7,0) node {$\boldsymbol{\cdot}$};
\draw (.7,-.3) node {$s_0$};
\draw (0,1) node {$\boldsymbol{\cdot}$};
\draw (-.3,1) node {$\frac{c}{\lambda}$};
\draw (2,0) node {$\boldsymbol{\cdot}$};
\draw (2,-.3) node {$s_1$};
\draw (1.5,0) node {$\boldsymbol{\cdot}$};
\draw (1.5,-.3) node {$s_2$};
\draw plot [smooth,tension=.6] coordinates {(-2,.2) (0,.4) (.8,1.75) (2,1) (3.5,.2) (4.5,.25) (5,.1) (6,.05)};
\draw (-1.7,.5) node {$w_R(s)$};
\draw [dotted] (0,1)--(2,1);
\draw (1.5,1.38) node {$\bullet$};
\draw [thick, dashed] (.725,2.0465)--(3.425,-.2075);
\end{scope}
\end{tikzpicture}
\end{center}
In order to prove \eqref{normalcomp1}, we essentially expand the Floer curve into a Fourier series and show that the coefficients, viewed as functions of the variable $s$, satisfy \eqref{ode} and use this bound to show that the $C^{m-1}$-norms of $\uperp$ go to zero uniformly in $k$. 
\begin{proposition}
\label{normcompprop2}
The $C^{m-1}$ norm of the normal component $\widetilde{u}^{k,\ell}_\perp$ converges to zero as $\ell\to\infty$, that is
\[
\sup_{k\geq\ell}\left\|\widetilde{u}^{k,\ell}_\perp\right\|_{C^{m-1}}\to0\quad\mathrm{as}\quad \ell\to\infty.
\]
\end{proposition}
\begin{proof}
Consider the space $L^2_{\phi^A_{T}}(\R,\Hil)$ of $\phi^ A_T$-periodic maps
\[
L^2_{\phi^A_T}(\R,\Hil):=\left.\left\{u\in L^2(\R,\Hil)\;\right\vert u(t+T)=\phi^A_{-T}u(t)\right\}
\]
and acting on it the densely defined operator $-i\partial_t$. Using the fact that the maps $z_n$ are a complete eigenbasis of $\phi^A_T$ with eigenvalues $e^{iaTn^d}$, we observe that the space $L^2_{\phi^A_T}$ has a complete basis of eigenfunctions $u_{p,n}$ of $-i\partial_t$ with eigenvalues $\lambda_{p,n}$ given by
\[
u_{p,n}(t)=e^{i(\frac{2\pi}{T}p-an^d)t}z_n;\quad\lambda_{p,n}=\frac{2\pi}{T}p-an^d
\]
for $p,n\in\Z$. Even though $\lambda_{p,n}\neq0$ for all $p,n\in\Z$, there exists sequences for which $(\lambda_{p',n'})\to0$, or 
\[
\inf_{p,n\in\Z}\;\abs{\frac{2\pi}{T}p-an^d}=0.
\]
We overcome this \emph{small divisor problem} by using the assumption that the number $\frac{aT}{2\pi}$ is Diophantine with irrationality measure $r<\infty$ as follows: for fixed $n\in\N$, we have the bound
\begin{align}
\label{sdpbound1}
\inf_{p\in\Z}\;\abs{\frac{2\pi}{T}p-an^d}=\frac{2\pi n^d}{T}\inf_{p\in\Z}\;\abs{\frac{aT}{2\pi}-\frac{p}{n^d}}\geq \frac{c}{n^{d(r-1)}}
\end{align}
for some $c>0$.\par
We now view a Floer curve $\wt{u}^k$ as a map 
\[
\wt{u}^k:\R\to L^2_{\phi^A_T}(\R,\C^{2k+1})\subset L^2_{\phi^A_T}(\R,\Hil)
\]
satisfying
\[
\partial_s\wt{u}^k=-i\partial_t\wt{u}^k-\varphi_k(s)\nabla G_t^k(\wt{u}^k)
\]
and the asymptotic conditions $\wt{u}^k(s,\cdot)\to0$ as $s\to\pm\infty$. Expanding the $s$-evaluation as a Fourier series with respect to $n\in\Z$, $p\in\Z-an^dT/(2\pi)$,
\[
\wt{u}^k(s,t)=\wt{u}^k(s)(t)=\sum_{n=-k}^k\sum_{p}\widehat{\wt{u}^k(s)}(p,n)e^{i2\pi pt/T}z_n
\]
with $\widehat{\wt{u}^k(s)}:\Z\times\Z\to\C$, we obtain $s$-dependent sequences $w^k_{p,n}(s)=\widehat{\wt{u}^k(s)}(p,n)$ which satisfy
\[
(w_{p,n}^k)'(s)=\lambda_{p,n}w_{p,n}^k(s)+f_{p,n}^k(s);\quad w_{p,n}^k(s)\to0\text{ as }s\to\pm\infty,
\]
where $f_{p,n}^k(s):=-\widehat{\nabla G_t^k(\wt{u}^k)(s)}(p,n)\in\C$ are the Fourier coefficients of $\nabla G^k(\wt{u}^k)(s)$. Here we view $\nabla G^k(\wt{u}^k)$ as a map from $\R$ to $L^2_{\phi^A_T}(\R,\Hil)$ by
\[
\nabla G^k(\wt{u}^k)(s)(t):=\varphi_k(s)\nabla G^k_t(\wt{u}^k(s,t))
\]
so that
\[
\nabla G^k(\wt{u}^k)(s)(t)=\varphi_k(s)\sum_{n=-k}^k\sum_{p}\widehat{\nabla G^k(\wt{u}^k)(s)}(p,n)e^{i2\pi pt/T}z_n.
\]
Since $\nabla G^k(\wt{u}^k)(s)$ is $m=\floor{h/d}$-times continuously differentiable with respect to time and has uniformly bounded derivatives, and by the decay property of the $h$-regularizing nonlinearity, we know that
\begin{align}
\label{equationx}
%\abs{f_{n,p}^k(s)}&=\abs{\widehat{\nabla G^k(\wt{u}^k)(s)}(n,p)}<c_{\alpha,\delta}|n|^{-{\delta}-\frac{1}{2}}|p|^{-\alpha-\frac{1}{2}}\\
%\abs{f_{n,p}^k(s)}&=\abs{\widehat{\nabla G^k(\wt{u}^k)(s)}(n,p)}=o\pa{|n|^{-{\delta}-\frac{1}{2}}|p|^{-\alpha-\frac{1}{2}}}\\
\op{f_{p,n}^k}_{C^0}|p|^m|n|^h&=\op{\widehat{\nabla G^k(\wt{u}^k)}(\cdot)(p,n)}_{C^0}|p|^m|n|^h\to0\text{ as }|p|,|n|\to\infty
\end{align}
where the $C^0$-norm is with respect to $s\in\R$. Note that here and below it is implicitly assumed that the limit is uniform with respect to $k\in\Z$, and since the argument $\wt{u}^k$ of $\nabla G^k$ also depends on $t$, we additionally have to use the result in \Cref{bootstrapping} that $\wt{u}^k$ is $\floor{h/d}$ times continuously differentiable and its derivatives are uniformly bounded in $s$ and $t$. Combining this with \eqref{sdpbound1} and \Cref{elementarylemma} %and the fact that we have uniform in $s$ bounds for $\nabla G^k(\wt{u}^k)$
we obtain
\begin{align}
\label{decayrateblah}
%\abs{w_{n,p}^k(s)}&=\abs{\widehat{\wt{u}^k(s)}(n,p)}< C_{\alpha,\delta}|n|^{\frac{1}{2}(d(r-1)-h-1)}|p|^{-\alpha-\frac{1}{2}}\\
%\abs{w_{n,p}^k(s)}&=\abs{\widehat{\wt{u}^k(s)}(n,p)}=o\pa{|n|^{h-d(r-1)-\frac{1}{2}}|p|^{-\alpha-\frac{1}{2}}}\\
\op{w_{p,n}^k}_{C^0}|p|^m|n|^{h-d(r-1)}&=\op{\widehat{\wt{u}^k}(\cdot)(p,n)}_{C^0}|p|^m|n|^{h-d(r-1)}\to0\text{ as }|p|,|n|\to\infty
\end{align}
where again the $C^0$-norm is with respect to $s\in\R$. \par
We now bound the time derivative which, together with the bound on the gradient of the nonlinearity, also leads to a bound of the $s$-derivative which concludes the proof. From
\[
\abs{\partial_t^{m-1}\wt{u}^{k,\ell}_\perp(s,t)}^2\leq\sum_{|n|=\ell+1}^k\pa{\sum_{p}\abs{\widehat{\wt{u}^k(s)}(p,n)}\abs{p}^{m-1}}^2
\]
and the above it follows that the tail $\uperp$ for $k\geq\ell$ satisfies
\[
\op{\partial_t^{m-1}\uperp}_{C^0}=o(\ell^{-h+d(r-1)+1/2}).
\]
Let $\nabla^\ell_{\perp} G^k_t(u)$ denote the component of the gradient of $G^k_t(u)$ which is normal to the finite-dimensional subspace $\C^{2\ell+1}\subset\Hil$. Since $\op{\nabla^\ell_\perp G^k_t(\wt{u}^k)}_{C^{m-1}}$ goes to zero uniformly in $k\geq\ell$ as $\ell\to\infty$ by \eqref{equationx}, and since $\wt{u}^k$ satisfies the Floer equation, we obtain that the $s$-derivatives also go to zero uniformly in $k$, so that $\op{\uperp}_{C^{m-1}}$ goes to zero uniformly in $k$ as long as $h>dr> d(r-1)+\frac{1}{2}$.
\end{proof}
Since almost all numbers have $r=2$, generically this bound comes down to $h>2d$. In the case of the Schrödinger equation this means we need $h>4$ and for the wave equation $h>2$.

\section{Completing the proof}
\label{completingproof}
We now complete the proof of the \Cref{mainthm}. This consists of three parts: first, we prove convergence of the sequence (or a subsequence) of Floer curves $(\widetilde{u}^k)_k$ to a solution $\widetilde{u}$ of the Floer equation on the full Hilbert space. This is not immediate, since $\Hil$, or even the support of the nonlinearity in $\Hil$, is not compact, so that we cannot use Gromov-Floer compactness. We will prove this convergence in the $\cloc$-topology, where $m=\floor{h/d}\geq2$. \par
Secondly, we establish the asymptotic properties to conclude that this Floer curve connects the single (trivial) solution of the free Hamiltonian equation, to a (nontrivial) solution of the full Hamiltonian equation.\par %In essence, this shows that the sequence of curves $(\widetilde{u}^k)_k$ breaks.\par
Finally, we discuss the regularity of the solution. The regularity of the solution we find will, of course, depend on the regularity of the nonlinearity. % If the nonlinearity is strongly $A$-admissible, we will find a strong solution.\par
We stress here that the existence of the finite-dimensional Floer curves $\widetilde{u}^k$ for the finite-dimensional nonlinearities $G_t^k$, which make up the sequence $(\widetilde{u}^k)_k$, is proven in \Cref{fdsection}. 
\begin{theorem}
\label{finallemma}
There exists a subsequence of the sequence $(\widetilde{u}^k)_k$ of Floer curves $\widetilde{u}^k:\R\times\R\to\C^{2k+1}$ which $\cloc$-converges to a solution $\widetilde{u}:\R\times\R\to\Hil$ of the Floer equation 
\[
(\partial_s+i\partial_t)\widetilde{u}+\varphi(s)\nabla G_t(\widetilde{u})=0
\]
satisfying $\widetilde{u}(s,t+T)=\phi_{-T}^A\widetilde{u}(s,t)$.
\end{theorem}
\begin{proof}
By \Cref{normcompprop2} we know that the $C^{m-1}$-norms of $\uperp$ converge to zero as $k$ increases. 
%We first show that the normal component locally converges to zero in the $C^m$-norms as the dimension increases. For this, we use that from \Cref{lemma4} we know that
%\[
%\sup_{k\geq\ell}\left\|\nabla^\ell_{\perp} G^k\right\|_{C^m}\to0\quad\text{as}\quad\ell\to\infty.
%\]
%After restricting to any bounded open subset it follows from 
%\[
%%\left\|\nabla^{\ell}_{\perp} G^k(\widetilde{u}^k)\right\|_{W^{4,p}}\leq c_1\left(\cdot \left\|\nabla^{\ell}_{\perp} G^k\right\|_{C^4}+1\right)\left(\left\|\widetilde{u}^k\right\|_{W^{4,p}}+1\right)\\
%\left\|\nabla G^k_\perp(\widetilde{u}^k)\right\|_{W^{m,p}}\leq c_1\left\|\nabla G^k_\perp\right\|_{C^m}\left(\left\|\check{u}^k\right\|_{C^0}^{m-1}+1\right)\left(\left\|\check{u}^k\right\|_{W^{m,p}}+1\right)
%\] 
%as in the proof of \Cref{bootstrapping}, that we have
%\[
%\sup_{k\geq\ell}\left\|(\partial_s+i\partial_t)\widetilde{u}^{k,\ell}_\perp\right\|_{W^{m,p}}=\sup_{k\geq\ell}\left\|\nabla^\ell_\perp G_t^k(\widetilde{u}^k)\right\|_{W^{m,p}}\to0\quad\text{as}\quad \ell\to\infty. 
%\]
%Since the $L^p$-norm of $\widetilde{u}^{k,\ell}_\perp$ locally converges to zero uniformly in $k\geq\ell$ by \Cref{normcompprop}, elliptic regularity for the Cauchy-Riemann operator then gives us 
%\[
%\sup_{k\geq\ell}\left\|\widetilde{u}^{k,\ell}_\perp\right\|_{W^{m+1,p}}\to0\quad\text{as}\quad \ell\to\infty,
%\]
%and so
%\[
%\sup_{k\geq\ell}\left\|\widetilde{u}^{k,\ell}_\perp\right\|_{C^m}\to0\quad\text{as}\quad \ell\to\infty.
%\]
%\par
To show that the limit of $(\wt{u}^k)_k$ exists, we start with the observation that there is a  subsequence of $(\widetilde{u}^{k,\ell})_k$ of maps from $\R\times\R$ to $\C^{2\ell+1}$ which $\cloc$-converges to a smooth map $\widetilde{u}^\ell:\R\times\R\to\C^{2\ell+1}$ as $k\to\infty$ for all $\ell$. We stress that the maps $\widetilde{u}^{k,\ell}$ take values in $\C^{2\ell+1}$, so that compactness holds by analogous reasons as for finite-dimensional nonlinearities. In particular, by the bounded support condition in \Cref{admissible}, the maximum principle ensures that the image is in a ball of radius $R_\ell\subset\C^{2\ell+1}$. Because we have locally bounded $W^{m+1,p}$-norms and hence, by elliptic bootstrapping and passing to a diagonal subsequence, local $W^{m,p}$-convergence, by Sobolev embedding we also have local convergence in the $C^{m-1}$-norm.
%To see this, let us fix some $\ell_0$. Then there is a subsequence $(\widetilde{u}^{k_i,\ell_0})_{k_i}\subseteq(\widetilde{u}^{k,\ell_0})_k$ which converges to a smooth map $\widetilde{u}^{\ell_0}:\R\times\R\to\C^{\ell_0}$ as $k\to\infty$. Let $\ell_1=\ell_0+1$, then there is a subsequence of $(\widetilde{u}^{k_i,\ell_0})_{k_i}$ which converges to a smooth map $\widetilde{u}^{\ell_1}:\R\times\R\to\C^{\ell_1}$. Continuing in this fashion, we get a sequence of sequences $((\widetilde{u}^{k_i,\ell_n})_{k_i})_n$, each sequence being a subsequence of the previous one, and each converging to a map $\widetilde{u}^{\ell_n}:\R\times\R\to\C^{\ell_n}$. Taking the diagonal then yields a sequence which converges for all given $\ell$ to a Floer curve taking values in $\C^\ell$. 
Passing to a diagonal subsequence yet again, we obtain $\cloc$-convergence for all $\ell$ simultaneously.\par
After restricting to any bounded open subset, we now show that the sequence of maps $(\widetilde{u}^k)_k$ thus obtained is Cauchy in the $C^{m-1}$-norm, which is sufficient to prove $\cloc$-convergence. Let $\epsilon>0$. Then there is an $\ell$ such that $\sup_{k\geq\ell}\left\|\widetilde{u}^{k,\ell}_\perp\right\|_{C^{m-1}}<\epsilon/3$. For this $\ell$, the sequence $(\widetilde{u}^{k,\ell})_k$ converges to $\widetilde{u}^\ell$, so there is $k_0\geq\ell$, so that for $k,k'\geq k_0$ we have $\left\|\widetilde{u}^{k,\ell}-\widetilde{u}^{k',\ell}\right\|_{C^{m-1}}<\epsilon/3$. Hence
\[
\left\|\widetilde{u}^k-\widetilde{u}^{k'}\right\|_{C^{m-1}}\leq\left\|\widetilde{u}^{k,\ell}_\perp\right\|_{C^{m-1}}+\left\|\widetilde{u}^{k,\ell}-\widetilde{u}^{k',\ell}\right\|_{C^{m-1}}+\left\|\widetilde{u}^{k',\ell}_\perp\right\|_{C^{m-1}}<\epsilon.
\]
\end{proof}
%We can now prove the existence of a convergent subsequence for any sequence of Floer curves. For this, we use \Cref{finallemma} and a diagonal subsequence argument.
%\begin{theorem}
%\label{compactness}
%Any sequence $(\wt{u}_n,\tau_n)$ of pairs, with $\wt{u}_n:\R\times\R\to\Hil$ satisfying
%\[
%\overline{\partial}\wt{u}_n+\varphi_{\tau_n}(s)\nabla G_t(\wt{u}_n)=0,\qquad \wt{u}_n(s,t+T)=\phi^A_{-T}\wt{u}_n(s,t)
%\]
%and $(\tau_n)_n$ any sequence in $\R_{\geq0}$, has a convergent subsequence.
%\end{theorem}
%\begin{proof}
%Let $\wt{u}_n$ be any Floer curve $\wt{u}_n:\R\times\R\to\Hil$. This curve is the $\cloc$-limit of its finite-dimensional approximations $(\wt{u}_{n,\ell},\tau_{n})_\ell$ where $\wt{u}_{n,\ell}:=\wt{u}_n\vert_{\C^{2\ell+1}}$. If $\wt{u}_n$ takes values in some finite-dimensional subspace of $\Hil$, this is obvious. If not, then \Cref{finallemma} shows it can be obtained as the $\cloc$-limit of its finite-dimensional approximations. We can now find a diagonal subsequence of $(\wt{u}^k_n)_{k,n}$ to obtain
%\[
%\op{\wt{u}^k-\wt{u}^{k'}}\leq\op{\wt{u}^k-\wt{u}^k_n}+\op{\wt{u}^k_n-\wt{u}^{k'}_n}+\op{\wt{u}^{k'}-\wt{u}^{k'}_n}
%\]
%where the first and last term are small for $n$ large enough due to \Cref{finallemma} and the second term is small by elliptic bootstrapping and after passing to a subsequence for which $(\tau_k)_k$ converges.
%\end{proof}
Let us now establish the asymptotic behaviour of the Floer curve. Specifically, we have
\begin{theorem}
\label{lemmax}
Using finiteness of energy, the limit Floer curve $\widetilde{u}:\R\times \R\to\Hil$ satisfies the following asymptotic conditions: there exists sequences $s_n^\pm\in\R$ with $s_n^\pm\to\pm\infty$ as $n\to\infty$ such that
\begin{align*}
\lim_{n\to\infty}\widetilde{u}(s_n^-,t)=u_0(t),\qquad\lim_{n\to\infty}\widetilde{u}(s_n^+,t)=u_1(t),
\end{align*}
in the $C^{m-1}$-sense. Here $u_0=0$ is the trivial and only fixed point of the free flow and $u_1$ is a $\phi_T^A$-periodic orbit of $G_t$.
\end{theorem}
\begin{proof}
Because the energy is bounded in terms of the $C^0$-norm of $G$ (see \Cref{finallemma}), we get
\[
E(\widetilde{u})=\int_{-\infty}^\infty\int_0^T\left|\partial_t\widetilde{u}(s,t)-\varphi(s)X_t^{G}(\widetilde{u}(s,t))\right|^2dt\;ds\leq4T\left\|F\right\|_{C^0}.
\]
%Let $\gamma\in\N$ and consider the curve $\widetilde{u}^\gamma:\R\times\R\to\C^{2\gamma+1}$. 
Choose sequences $s_\gamma^\pm\in\R$ with $\gamma\leq s_\gamma^+\leq2\gamma$ and $\gamma\leq -s_\gamma^-\leq2\gamma$ such that 
\[
\gamma\int_0^T\left|\partial_t\widetilde{u}(s_\gamma^\pm,t)-\varphi(s_\gamma^\pm)X_t^{G}(\widetilde{u}(s_\gamma^\pm,t))\right|^2dt
\]
is bounded by 
\[
\int_\gamma^{2\gamma}\int_0^T\left|\partial_t\widetilde{u}(s,t)-\varphi(s)X_t^{G}(\widetilde{u}(s,t))\right|^2dt\;ds
\]
or 
\[
\int_{-2\gamma}^{-\gamma}\int_0^T\left|\partial_t\widetilde{u}(s,t)-\varphi(s)X_t^{G}(\widetilde{u}(s,t))\right|^2dt\;ds
\]
respectively. This implies
\[
\int_0^T\left|\partial_t\widetilde{u}(s_\gamma^\pm,t)-\varphi(s_\gamma^\pm)X_t^{G}(\widetilde{u}(s_\gamma^\pm,t))\right|^2dt\leq \frac{4T\left\|F\right\|_{C^0}}{\gamma}\to0\text{ as }\gamma\to\infty.
\] 
Now we write $\widetilde{u}=(\widetilde{u}^{\ell},\widetilde{u}^{\ell}_\perp):\R\times\R\to\C^{2\ell+1}\oplus\Hil/\C^{2\ell+1}$ for $\ell\in\N$. Since by the maximum principle $\widetilde{u}^{\ell}$ takes values in $B_{R_\ell}(0)\subset\C^{2\ell+1}$, after passing to a subsequence we can assume that $\widetilde{u}^{\ell}(s_{\gamma}^\pm,\cdot)$ $C^{m-1}$-converges as $\gamma\to\infty$. After passing to a diagonal subsequence we can assume $\widetilde{u}^{\ell}(s_\gamma^\pm,\cdot)$ $C^{m-1}$-converges for all $\ell$ simultaneously. Since $\|\widetilde{u}^{\ell}_\perp\|_{C^{m-1}}\to0$ by \Cref{normcompprop2} and \Cref{finallemma}, we have that $\widetilde{u}(s_{{\gamma}}^\pm,\cdot)$ $C^{m-1}$-converges, that is
\[
\lim_{\gamma\to\infty}\widetilde{u}(s_\gamma^-,t)=u_0(t),\qquad\lim_{\gamma\to\infty}\widetilde{u}(s_\gamma^+,t)=u_1(t)
\]
which both satisfy the Hamiltonian equation \eqref{phiform}. Because $\varphi(s_\gamma^-)=0$, the solution $u_0(t)$ is the trivial solution. Because $\varphi(s_\gamma^+)=1$ we have indeed found a solution $u_1$ to \eqref{phiform}.
\end{proof}
Pictorially, the limit looks like the breaking
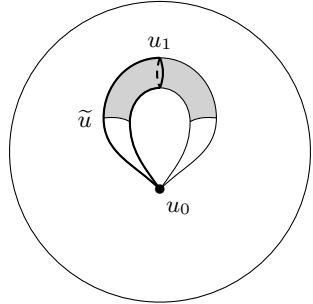
\begin{center}
\begin{tikzpicture}
\draw (0,.5) circle (2cm);
\draw (0,0) node {$\bullet$};
\draw (.25,-.25) node {$u_0$};
\draw (0,0) to[out=45,in=270] (.75,.95);
\draw (0,0) to[out=135,in=270] (-.75,.95);

\draw (0,0) to[out=55,in=270] (.4,.9);
\draw (0,0) to[out=125,in=270] (-.4,.9);

\fill[fill=gray!35!white] (.75,.95) to[out=90,in=0] (0,1.75) to[out=180,in=90] (-.75,.95) plot [smooth,tension=1] coordinates {(-.75,.95) (-.55,.95) (-.4,.9)} to[out=90,in=180] (0,1.35) to[out=0,in=90] (.4,.9) to[out=15,in=185] (.55,.95) to (.75,.95);
\draw (.75,.95) to[out=90,in=0] (0,1.75) to[out=180,in=90] (-.75,.95) plot [smooth,tension=1] coordinates {(-.75,.95) (-0.55,.95) (-.4,.9)} to[out=90,in=180] (0,1.35) to[out=0,in=90] (.4,.9) plot [smooth,tension=1] coordinates {(.4,.9) (.55,.95) (.75,.95)};

\draw [thick] (0,1.75) to[out=-70,in=70] (0,1.35);
\draw (0,1.75) [thick, densely dashed] to[out=-110,in=110] (0,1.35);

\draw (-1,.95) node {$\widetilde{u}$};
\draw (0,1.95) node {$u_1$};

\draw [thick] (0,0) to[out=135,in=270] (-.75,.95) to[out=90,in=180] (0,1.75);
\draw [thick] (0,0) to[out=125,in=270] (-.4,.9) to[out=90,in=180] (0,1.35);
\end{tikzpicture}
\end{center}
%We stress that the Floer curve cannot leave the (bounded) support of the nonlinearity because of the maximum principle. 
Since there is no other fixed point of the free flow than the trivial solution, we indeed find a nontrivial fixed point of the full flow, provided that $\nabla F_t(0)\neq0$. \\\par
%\begin{lemma}
%
%\end{lemma}
%\begin{proof}
%
%\end{proof}
%We finish the proof by showing that we have found a non-trivial solution $u_1\in\Hil$.
We finish by discussing the regularity of the solution.
\begin{theorem}
\label{regularitythm}
The Floer curve $\wt{u}$, and in particular the $T$-periodic solution $u(t)=\phi_{t}^Au_1(t)$ we obtain from the $\phi_T^A$-periodic solution $u_1(t)$ found in \Cref{lemmax}, is of regularity $h-d(r-1)-\frac{1}{2}>0$ for every $h> dr$, i.e. $\wt{u}:\R\times\R\to\Hil_{h-d(r-1)-1/2}\subset\Hil$. 
%The solution $u_1:\R\times\R\to\C$ of the Hamiltonian PDE $\partial_tu_1=X_t^G(u_1)$ is a weak (in the $x$-variable) solution $u_1\in\Hil$ and it is of class $C^\infty$ in the $t$-variable. %If $F_t$ is strongly $A$-admissible, the solution is strong.
\end{theorem}
\begin{proof}
%We will prove the strong result, the weak result will follow by setting $b=0$. 
Let $F_t$ be $A$-admissible. From the proof of \Cref{normcompprop2} equation \eqref{decayrateblah} we know that
%\[
%\sup_{k\geq\ell}\left\|\widetilde{u}^{k,\ell}_\perp\right\|_{C^0} =o(\ell^{-\frac{s}{3}+\frac{2}{3}+\frac{4d(r-1)}{3}})=o(\ell^{-\frac{1}{3}(s-s_0)}).
%\]
%\[
%\sup_{k\geq\ell}\left\|\widetilde{u}^{k,\ell}_\perp\right\|_{C^0}\ell^{d(r-1)-h}\to0\text{ as }\ell\to\infty.
%\]
%Writing the limit curve as $\widetilde{u}=(\widetilde{u}^\ell,\widetilde{u}^\ell_\perp)$, it follows that
%\[
%\left\|\widetilde{u}^{\ell}_\perp\right\|_{C^0} =o(\ell^{-\frac{s}{3}+\frac{2}{3}+\frac{4d(r-1)}{3}})=o(\ell^{-\frac{1}{3}(s-s_0)}).
%\left\|\widetilde{u}^{\ell}_\perp\right\|_{C^0}\leq c\ell^{-\frac{1}{3}(h-h_0)}.
%\]
%This means that 
the coefficients in the Fourier expansion of the Floer curve $\wt{u}$ satisfy
\[
\abs{\widehat{\wt{u}(s)}(p,n)}|n|^{h-d(r-1)}|p|^m\to0\text{ as }|n|,|p|\to\infty
\]
with $m=\floor{h/d}\geq 2$ uniformly for all $s$, which implies that 
\[
\abs{\wt{u}(s,t)}_{h-d(r-1)-1/2}^2\leq\sum_{n}\pa{\abs{n}^{h-d(r-1)-1/2}\sum_{p}\abs{\widehat{\wt{u}(s)}(p,n)}}^2
\]
is uniformly bounded for all $(s,t)\in\R\times\R$, where we sum over $n\in\Z$, $p\in\Z-an^dT/(2\pi)$. In particular, this holds as we let $s$ go to infinity, so that we obtain the same regularity for the non-trivial solution $u_1$. Subsequently, the solution $u(t)=\phi^A_tu_1(t)$ also has the same regularity since $\phi^A_t$ preserves Hilbert scales.
%
%
%\[
%\widehat{u_1(t,\cdot)}(\pm n)n^\delta\to0\quad\text{as}\quad n\to\infty\quad\text{for}\quad \delta<\frac{1}{3}(h-h_0) 
%\]
%which is equivalent to saying that the weak partial derivatives $\partial^\delta_xu_1(t,x)$ exist for all such $\delta$ and all $(t,x)\in\R\times\R$ in $\Hil$. We are interested in the regularity of the $T$-periodic solution $u(t)=\phi^A_{t}u_1(t)$ in the space variable, but this is the same as the regularity in the space variable of the $\phi_T^A$-periodic solution $u_1$ because $\phi_t^A$ preserves the Hilbert scales. 
\end{proof}

We again stress that for generic time period $T$, the irrationality measure is $r=2$, so that $h>2d$. The regularity of the solution depends on $h$ but also on the specific Hilbert space on which the Hamiltonian PDE is modeled. Let us apply our results to our two examples.
\begin{proposition}
\label{nlwnls3}
Viewing it as a PDE in the variables $s$, $t$ and $x$ with asymptotic conditions, the Floer equation
\[
\overline{\partial}\doubletilde{u}(s,t,x)+A\doubletilde{u}(s,t,x)+\varphi(s)\nabla F_t(\doubletilde{u}(s,t,x))=0
\]
with $A$-admissible nonlinearities admits a strong $(T,X)$-periodic solution
\[
\doubletilde{u}(s,t+T,x)=\doubletilde{u}(s,t,x)=\doubletilde{u}(s,t,x+X),\quad (s,t,x)\in\R\times\R\times\R
\]
for generic $T$ when $h>2\frac{1}{2}$ for the nonlinear wave equation and $h>5$ for the nonlinear Schrödinger equation. %Observe that we can view the Floer equation as a PDE in the variable $s$, $t$ and $x$ because $A$ is a differential operator in $x$. 
\end{proposition}
\begin{proof}
%Since the Floer curve is of class $C^3$ in the time variable, all time derivatives up to order three exist and we only have to consider the regularity of the solution in the space coordinate.\par
We define $\doubletilde{u}(s,t):=\phi^A_t\wt{u}(s,t)$, $(s,t)\in\R\times\R$ and subsequently view $\doubletilde{u}$ as a function of $s$, $t$ and $x$. \sloppy Recall that the Hilbert space for the nonlinear wave equation is $\Hil=W^{\frac{1}{2},2}\times W^{\frac{1}{2},2}$. In Hilbert scale notation we have $(\Hil)_k=W^{\frac{1}{2}+k,2}\times W^{\frac{1}{2}+k,2}$. By the Sobolev embedding theorem we have $W^{\frac{1}{2}+k,2}\subset C^{k}$. Since $A$ is of order $1$, we need our solution $\doubletilde{u}$ to be an element of $W^{\frac{1}{2}+1,2}\times W^{\frac{1}{2}+1,2}$ in order for it to be in $C^1\times C^1$. For generic time period $T$ the irrationality measure of $aT/2\pi$ is $r=2$, so for the solution to land in $\Hil_1=W^{\frac{1}{2}+1,2}\times W^{\frac{1}{2}+1,2}\subset C^1\times C^1$ and be a strong solution to the Floer equation, we need $h>2\frac{1}{2}$. %However, since we would like to solve the actual nonlinear wave equation, we need $\doubletilde{u}\in\Hil_2\subset C^2\times C^2$. In order to achieve this, we need $h>12$. 
Then {$\doubletilde{u}=(\doubletilde{\varphi},\doubletilde{\pi}):\R\times\R\times\R\to\R$} satisfies
\[
\kolomtwee{\partial_s\doubletilde{\varphi}-\partial_t\doubletilde{\pi}}{\partial_s\doubletilde{\pi}+\partial_t\doubletilde{\varphi}}=\chi\pa{\abs{\doubletilde{u}}_{-h}}{\varphi}(s)\kolomtwee{B\doubletilde{\varphi}-B^{-1}\partial_1g_t(\doubletilde{\varphi}*\psi)*\psi-B^{-1}c_t}{B\doubletilde{\pi}}
\]
with cut-off function $\chi:[0,R]\to[0,1]$ to make the nonlinearity $A$-admissible. \par
The Hilbert space for the nonlinear Schrödinger equation is $L^2$. We have $(L^2)_k=W^{k,2}$. In order to get a strong solution $\doubletilde{u}$ to the Floer equation with nonlinear Schrödinger type Hamiltonian, we need $\doubletilde{u}$ to be of class $C^2$ in the spatial variable. For generic time period $T$ the irrationality measure of $aT/2\pi$ is $r=2$ and so when the $A$-admissible nonlinearity $F_t$ is $h$-regularizing with $h>5$, we have $\doubletilde{u}(s,t)\in\Hil_{2+\frac{1}{2}}=W^{2+\frac{1}{2},2}\subset C^2$ so that we get a strong solution of
\[
\partial_s\doubletilde{u}+i\partial_t\doubletilde{u}=-\partial_x^2\doubletilde{u}+\chi\pa{\abs{\doubletilde{u}}_{-h}}\varphi(s)\partial_1f\pa{\abs{\doubletilde{u}*\psi}^2,x,t}\pa{\doubletilde{u}*\psi}*\psi.
\]
Since the Floer curve $\wt{u}$ is of class $C^{m-1}$ in the $s$- and $t$-variables, all $s$- and $t$-derivatives of $\doubletilde{u}(s,t):=\phi^A_t\wt{u}(s,t)$ up to order $m-1\geq1$ exist. Note that differentiability in the time coordinate of the $\phi_T^A$-periodic solution itself does not immediately imply differentiability in $t$ for the corresponding $T$-periodic solution of the Floer equation. This is because the $t$-derivative of $\phi_{t}^A$ involves $JA$ which decreases regularity. So our results do \emph{not} follow from elliptic regularity. More specifically, the time derivative of $\doubletilde{u}$ is given by
\[
\frac{d}{dt}\doubletilde{u}=\phi_{t}^A\left(\frac{d}{dt}\wt{u}\right)+\left(\frac{d}{dt}\phi_{t}^A\right)\wt{u}.
\]
Since $\wt{u}$ is of class $C^{m-1}$ in the time variable, the first term is sufficiently regular. For the second term, recall that $\frac{d}{dt}\phi_{t}^A=JA\phi_{t}^A$ and so the second term only changes the regularity of $\wt{u}$ with respect to the space variable by decreasing it by $d$. In particular, the regularity in time coordinate depends on the regularity in the space coordinate. However, since above we gave conditions to ensure that we have enough regularity in the space variable, that is, $\doubletilde{u}(s,t)\in\Hil_d=\text{Dom}(A)$, the time derivatives in the strong sense exist as well. Observe that the regularity requirements stated above ensure that the single $s$-derivative also exists. Finally we remark that by \Cref{lemmax} and \Cref{regularitythm} the asymptotics of the Floer curve have the same $t$- and $x$-regularity as the Floer curve itself.
%
%
%all time derivatives of $u_1(t)$ up to order three exist. Note that smoothness in the time coordinate of the $\phi_T^A$-periodic solution does not imply smoothness in $t$ for the corresponding $T$-periodic solution of the Hamiltonian PDE. This is because the $t$-derivative of $\phi_{t}^A$ involves $JA$ which is only densely defined. So our results do \emph{not} follow from elliptic regularity and the Hamiltonian PDE itself. More specifically, the time derivative of $u(t)$ is given by
%\[
%\frac{d}{dt}u(t)=\phi_{t}^A\left(\frac{d}{dt}u_1(t)\right)+\left(\frac{d}{dt}\phi_{t}^A\right)u_1(t).
%\]
%Since $u_1(t)$ is of class $C^3$ in the time variable, the first term is sufficiently regular. For the second term, recall that $\frac{d}{dt}\phi_{t}^A=JA\phi_{t}^A$ and so the second term only changes the regularity of $u_1(t)$ with respect to the space variable by decreasing it by $d$. However, since above we gave conditions to ensure that we have enough regularity in the space variable, the time derivatives in the strong sense exist as well. 
\end{proof}

\section{Periodic solutions for Hamiltonian PDEs}
\label{subqsection}
As a corollary to the existence of a fixed point of $\phi_T^H$ for a Hamiltonian with $A$-admissible nonlinearity with, in particular, bounded support in our weaker sense of \Cref{admissible}, we can now prove the existence of a fixed point when the nonlinearity is only weakly $A$-admissible. We want to stress here that we do not claim that these result could not be obtained using different methods and we rather include this as an application of our compactness result. We remark that there has been a significant amount of research on the problem of finding time-periodic solutions of Hamiltonian PDEs, e.g. \cite{breziscoronnirenberg}, \cite{craigwayne}, \cite{kuksin1987}, \cite{wayne1990} and \cite{rabinowitz} to mention just a few; we refer to the comprehensive book \cite{berti} for an overview of the current state in the field. In particular, a KAM result was proven in \cite{kuksin2015kam} and \cite{kuksin2016kam} for the Schrödinger equation with regularizing nonlinearity that we consider. Note that the small divisor problem and regularization also play a key role in their considerations. The existence of time-periodic solutions was proven when the nonlinearity is time-independent or when it has a prescribed time dependence, for example, in \cite{gentileprocesi2008} and \cite{gentileprocesi2009}. We want to stress that we are studying general nonautonomous Hamiltonian PDE without any predescribed time-behaviour of the nonlinearity.\\\par

The idea, now, is that given a weakly $A$-admissible nonlinearity $\widetilde{F}_t$ we compose it with a cut-off function $\chi$ to get an $A$-admissible nonlinearity $F_t$. We then show that when the support of $\chi$ is sufficiently large, the region where a possible $T$-periodic solution could exist stays away from the cut-off region. The \Cref{mainthm} then implies that there exists a periodic solution for the Hamiltonian with this $A$-admissible nonlinearity $F_t$. Since this solution remains in the region where $\chi=1$, that is, where $F_t=\widetilde{F}_t$, we find that the solution is also a solution for the Hamiltonian with weakly $A$-admissible nonlinearity $\widetilde{F}_t$. 
%
%
%The idea, now, is that for a weakly $A$-admissible nonlinearity the free term $H_A$ of the Hamiltonian $H$ dominates $F_t$ when $|u|$ is large. This implies that the region in which a fixed point of $\phi_T^H$ \emph{could} exist, is bounded. From this we infer that a fixed point exists for the flow of $H$ with weakly $A$-admissible nonlinearity $\widetilde{F}_t$: if the periodic orbit for $\phi^H_T$ with $A$-admissible nonlinearity $F_t$ stays inside the region where $\varphi=1$, then it is also a periodic orbit for $\phi^H_T$ with weakly $A$-admissible nonlinearity $\widetilde{F}_t$. \\\par
%To prove this, we first bound the distance $\left|\phi_T^Au-u\right|$ from below by the $\Hil_{-h}$-norm of $u\in\Hil$.
\begin{lemma}
\label{freeflow}
Let $A$ be admissible and of degree $d$, let $(T,X)$ be admissible and let $h> dr$. Then there exists a positive $c\in\R$ such that
\[
\left|\phi_T^Au-u\right|^2\geq c|u|_{-h}^2.
\]
\end{lemma}
\begin{proof}
This is a similar occurrence of the small divisor problem as we have already seen:
\[
\left|\phi_T^Au-u\right|^2&=\sum_{n=0}^\infty\left|e^{ia(\pm n)^dT}-1\right|^2\left|\hat{u}(\pm n)\right|^2\\
&\geq c\sum_{n=0}^\infty n^{-2d(r-1)}\left|\hat{u}(\pm n)\right|^2\\
&=c\left|u\right|_{-d(r-1)}^2.
\]
where in the second line we use the small angle approximation and Diophantineness condition to write
\[
\abs{e^{ian^dT}-1}\approx \inf_{p\in\Z}\abs{an^dT-2\pi p} \geq 2\pi\frac{c'}{n^{d(r-1)}}
\]
similar to the proof of \Cref{normcompprop2}. Since $h> dr$ and $|u|_h<|u|_i$ whenever $h<i$, the result follows.
\end{proof}

\begin{theorem}
\label{mainthmsection1}
For a Hamiltonian PDE with weakly $A$-admissible nonlinearity there exists a forced time-periodic solution which is of regularity $h-d(r-1)-\frac{1}{2}$ for $h> dr$, %for every $s>s_1>s_0$, 
that is, $u:\R\to\Hil_{h-d(r-1)-1/2}\subset\Hil$ with
\[
\partial_tu=JAu+J\nabla F_t(u),\qquad u(t+T)=u(t).
\]
\end{theorem}
\begin{proof}
Choose a cut-off function $\chi^R:\R_{\geq0}\to[0,1]$ which equals $1$ on $[0,R]$, is $0$ on $[R+1,\infty)$ and has slope $-2\leq(\chi^R)'(r)\leq0$ for $r\in[R,R+1]$. Defining $F_t=F_t^R$ as in \Cref{cutoffprop} using $\chi^R$, it follows that ${F}_t:\Hil_{-h}\to\R$, and hence also when viewed as a map ${F}_t:\Hil\to\R$, has bounded first derivatives, independent of $R$. Here we use that $\widetilde{F}_t$ has bounded first derivatives, even when $c_t\neq 0$ in \Cref{admissible}. Since therefore $G_t$ has bounded first derivatives with respect to $u$, we have $|X_t^G(u)|\leq c'$ for some $c'>0$ which is independent of $R$, and hence $|\phi_T^G(u)-u|\leq c'T$. Since $|\phi_T^Au-u|\geq\sqrt{c}|u|_{-h}$, it follows that $u$ cannot be a fixed point whenever $|u|_{-h}>\frac{c'T}{\sqrt{c}}$: we have that $\phi_T^H=\phi_T^A\circ\phi_T^G$ and $\phi_T^A$ moves any $u$ a distance at least $\sqrt{c}|u|_{-h}$ away, while preserving the $\Hil_{-h}$-norm. %This latter statement follows from the fact that $\phi_t^A$ multiplies each Fourier mode by a complex number of norm one. 
However, $\phi_T^G$ only moves the point $\phi_{-T}^Au$ a distance at most $c'T<\sqrt{c}|u|_{-h}$, so $u$ cannot be a fixed point. In fact, this shows that the entire $\phi_T^A$-periodic solution $u_1$ stays inside the $\Hil_{-h}$-ball of radius $\frac{c'T}{\sqrt{c}}+\epsilon$ for any $\epsilon>0$. This continues to hold for the $T$-periodic solution $u$ because $\phi^A_t$ preserves the $\Hil_{-h}$-norm. 
%
% Since the solution $u$ has $\Hil_{-h}$-norm less than $\frac{c'T}{\sqrt{c}}$ it cannot leave the ball of radius $\frac{c'T}{\sqrt{c}}+c'T$. 
Now we choose $R=\frac{cT'}{\sqrt{c}}+\epsilon$.
%
%We now compose the weakly $A$-admissible nonlinearity with a smooth cut-off function (as in \Cref{cutoffprop}) which has bounded support and which equals $1$ on the interval $[0,\frac{c'T}{\sqrt{c}}+\epsilon]$. This gives us an $A$-admissible nonlinearity for which 
For this $A$-admissible nonlinearity the existence of a fixed point follows from the main theorem. By the above argument, this fixed point is also a fixed point of the time-$T$ flow of the Hamiltonian with weakly $A$-admissible nonlinearity $\widetilde{F}_t$ we started with, thus proving the theorem.
\end{proof}
We now show that when the nonlinearity is $h$-regularizing for all $h\in\N$, we find a periodic solution for almost all time periods, since Diophantine numbers have full measure, which is of class $C^\infty$ in both the time and spatial variable.
\begin{corollary}
Consider a Hamiltonian PDE with admissible $A$ and with $\infty$-regularizing $T$-periodic nonlinearity $\wt{F}_t$ with bounded $C^\alpha$-norms as in condition 3 of \Cref{admissible}. Then for admissible $(X,T)$ there exists a strong forced $T$-periodic solution which is smooth in both the time and space coordinate.
\end{corollary}
\begin{proof}
This does not follow immediately from \Cref{mainthmsection1}, since there is no complete norm on $\Hil_{-\infty}$. In order to prove that we still find a periodic orbit for the Hamiltonian PDE with $\infty$-regularizing weakly $A$-admissible nonlinearity $\wt{G}_t$, which is even smooth in both the time and space variable, compose $\wt{G}_t$ as above with a cut-off function $\chi(|\cdot|_{-h})$ for any finite $h> dr$ to obtain an $h$-regularizing $A$-admissible nonlinearity. Applying the above result we find a periodic solution $u(t)\in\Hil_{h-d(r-1)-1/2}$. Since it is a solution to the PDE with $\infty$-regularizing weakly $A$-admissible nonlinearity we started with, we can a posteriori show that $u(t)$ has image in $\Hil_\infty$ and that it is smooth with respect to $t$: First, since $\wt{G}_t$ is $\infty$-regularizing, by definition its gradient takes values in $\Hil_\infty$ so that
\[
\partial_tu=J\nabla\wt{G}_t(u)\in\Hil_\infty;\qquad u(t+T)=\phi^A_{-T}u(t),
\]
that is, the Fourier coefficients of $\partial_tu(t)$ decay exponentially fast, which in turn shows that the Fourier coefficients of $\phi^A_Tu(t)-u(t)$ have the same decay rate. Now the $n$-th component $u_n(t)=\widehat{u(t)}(n)\cdot z_n$ of $u(t)=\sum_n \widehat{u(t)}(n)\cdot z_n$ satisfies
\[
\abs{\phi^A_Tu_n(t)-u_n(t)}&=\abs{e^{iaTn^d}-1}|u_n(t)|\\
&\approx\inf_{p\in\Z}\abs{aTn^d-2\pi p}|u_n(t)|\\
&\geq\frac{c}{n^{d(r-1)}}|u_n(t)|
\]
so that $|u_n(t)|$ still decays exponentially fast with $|n|\in\N$, that is, $u(t)\in\Hil_\infty$ for all $t$. It remains to show that $t\mapsto u(t)$ is of class $C^\infty$, for which we again use the fact that it satisfies the Hamiltonian equation. Applying $\partial_t$ to both sides of $\partial_tu=X^G_t(u)$ and observing that $G$ is smooth in $t$ and both $u(t)$ and $\partial_tu(t)$ are in $\Hil_\infty$, we see that $\partial_t^2u(t)\in\Hil_\infty$. By repeatedly applying $\partial_t$ to both sides of the equation, it follows that $\partial_t^\alpha u(t)\in\Hil_\infty$ for all $\alpha\in\N$. 
\end{proof}

Recalling the fact that not only Diophantine numbers have full measure, but even those numbers with irrationality measure $r=2$, for generic $T$ %the irrationality measure of $aT/2\pi$ is $r=2$, so that 
we need $h> 2d$ and so
\begin{corollary}
Consider a Hamiltonian PDE with admissible $A$ and with $h$-regularizing time-periodic nonlinearity with bounded $C^\alpha$-norms as in condition $3$ of \Cref{admissible}. Then for generic time period $T$, there exists a (weak) forced $T$-periodic solution which is of regularity $h-d-\frac{1}{2}$ for $h>2d$. %When the nonlinearity is $\infty$-regularizing we obtain a strong solution which is smooth in both the time and space coordinate. %In particular, when $s>14$ we find a strong time-periodic solution to the nonlinear wave equation and when $s>20$ we find a strong time-periodic solution to the nonlinear Schrödinger equation, for generic time period. When $s=\infty$, we find a solution which is smooth in both the space and time coordinate. 
\end{corollary}
In particular, for our examples the main theorem provides us with the following results. Here we use the result from \Cref{nlwnls3} combined with \Cref{mainthmsection1}.
\begin{corollary}
The nonlinear wave equation
\[
\ddot{\varphi}-\varphi_{xx}-\partial_1g_t(\varphi*\psi,x)*\psi-c_t=0,\qquad \varphi=\varphi(t,x)=\varphi(t,x+X),\; x\in S^1=\R/X\Z
\]
with $\psi,c_t=c_{t+T}\in C^h$ and $g_{t+T}=g_t$ being bounded and having bounded derivatives, admits a strong $T$-periodic solution for generic $T$, provided that $h>3\frac{1}{2}$. When $\psi,c_t=c_{t+T}\in C^\infty$, the solution is smooth in both time and space coordinate.
\end{corollary}
The fact that $h>3$ suffices follows from \Cref{regularitythm} and the proof of \Cref{nlwnls3}, together with the observation that we need two spatial derivatives. \par
In order to see that one can only expect to find a periodic solution for generic $T$ for $h>0$ large enough, we emphasize that this can even be seen from a direct computation using Fourier series in the case when $g_t=0$.
\begin{remark}
Expanding $\varphi(t,x)$ and $c(t,x)=c_t(x)$ in terms of a Fourier series as in \Cref{counterexample}, it follows that the resulting Fourier coefficients satisfy the equation $$\pa{\frac{2\pi n}{X}-\frac{2\pi p}{T}}\pa{\frac{2\pi n}{X}+\frac{2\pi p}{T}}\hat{\varphi}(p,n)=\hat{c}(p,n).$$ Since for any subsequence $(p',n')\subset (p,n)_{p,n\in\Z}$ each of the two factors can only converge to zero like $n^{1-r}$ (and only one of the two factors is close to zero), it follows that for $c\in W^{h+\frac{1}{2},2}$, that is, $c=(c,0)\in\Hil_h$, we find a solution $\varphi\in W^{h+\frac{1}{2}-d(r-1),2}$, that is, $u=(\varphi,\pi)\in\Hil_{h-d(r-1)}\subset\Hil$ with $d=1$, provided that $h>dr$ . On the other hand, it also follows that we cannot expect to find a solution of higher regularity.
\end{remark}
Now let us turn to the nonlinear Schr\"odinger equation. 
\begin{corollary}
The nonlinear Schrödinger equation
\[
i\dot{u}+u_{xx}+\partial_1f_t\left(\vert u*\psi\vert^2,x\right)(u*\psi)*\psi=0,\qquad u=u(t,x)=u(t,x+X),\; x\in S^1=\R/X\Z
\]
with $\psi\in C^h$ and $\widetilde{f}_{t+T}=\widetilde{f}_t:(s,x)\mapsto f_t(|s|^2,x)$ being bounded and having bounded derivatives, admits a strong $T$-periodic solution for generic $T$, provided that $h>5$. When $\psi\in C^\infty$, the solution is smooth in both time and space coordinate.
\end{corollary}
%When in the above examples the smoothing kernel $\psi$ is of class $C^\infty$, we get a solution which is smooth in both coordinates. More generally, we have the following

%Another remarkable corollary is the following
%\begin{corollary}
%Assume $A$ is admissible and $F_t\equiv F$ is time-independent and $F:\Hil_{-h}\to\R$ with $h>2+4d$ has bounded $C^\alpha$-norms. Then there exists a $T$-periodic solution for almost every $T$.
%\end{corollary}
%\begin{remark}
%Note that contrary to the finite-dimensional case, this result does not follow from Morse theory and the existence of critical points (at least when $d>0$): the Hamiltonian $H$ is not a Morse function on $\Hil$ since $H_A$ is only densely defined. Note as well that unlike in finite dimensions, any sequence of $T_n$-periodic orbits with $T_n\to0$ need not converge. 
%\end{remark}
\begin{remark}
One could alternatively think about the admissibility condition for the periods $(X,T)$ as a condition on $X$: one could fix a time period $T$, so that for generic $X$ the number $aT/2\pi$ is Diophantine (with $r=2$). Since $a=(2\pi/X)^d$ in the two main examples, this means that for fixed $T$, the space period $X$ should be such that ${(2\pi)^{d-1}T}{X^{-d}}$ is Diophantine (with $r=2$). We stress the Diophantineness condition (with $r=2$) can explicitly be checked for any chosen pair $(X,T)$.
\end{remark}

\section{A cup-length estimate}
\label{cuplengthsection}
While our ultimate goal is to develop a full Floer homology theory for Hamiltonian PDEs with regularizing nonlinearities, we already give an example of a result which definitely needs pseudoholomorphic curve techniques and cannot be proven using more classical techniques such as in \cite{rabinowitz}: we consider the classical result by Schwarz \cite{Schwarz2} and use our results to prove a cup-length estimate for a Hamiltonian system on a phase space which is the product of linear symplectic Hilbert space with a closed symplectic manifold. \par
Let $M=(M,\omega_M)$ be a closed (finite-dimensional) symplectic manifold with vanishing second homotopy group, $\pi_2(M)=\{0\}$. Then $\widetilde{M}:=M\times\Hil$ is an infinite-dimensional symplectic Hilbert manifold equipped with the product symplectic form $\omega=\pi_M^*\omega_M+\pi_{\Hil}^*\omega_{\Hil}$ and with a scale structure given by $\widetilde{M}_h:=M\times\Hil_h$, $h\in\R$. Here $\pi_M:\widetilde{M}_h\to M$, $\pi_{\Hil}:\widetilde{M}_h\to\Hil_h$ denote the projection onto the first or second factor, respectively. \par 

Note that infinite-dimensional phase spaces of this form appear when performing symplectic reduction using a Hamiltonian action on $\Hil$ which is non-trivial only on finitely many components. Alternatively, they arise in Hamiltonian systems incorporating both Hamiltonian mechanics and Hamiltonian field theory. Indeed, generalizing the class of Hamiltonian particle-field systems that we introduce in \cite{FabertLamoree21}, consider a symplectic manifold $(B,\omega_B)$ with a foliation by Lagrangian submanifolds, which contains $(M,\omega_M)$ as a symplectic submanifold, as well as a symplectic vector bundle $E\to B$ over $B=(B,\omega_B)$. Let $\Hil=(\Hil,\omega_{\Hil})$ denote a symplectic Hilbert space of sections in this bundle which are constant along leaves, where the symplectic bilinear form $\omega_{\Hil}$ on $\Hil$ is defined using the symplectic structures on the fibres. Now consider time-periodic Hamiltonians $$H_t=H^A+F_t:M\times\Hil\to\R\,\,\textrm{with}\,\,F_t(u_M,u_{\Hil})=f_t(u_M,u_{\Hil}^\rho(u_M)),$$ where $u_{\Hil}\mapsto u_{\Hil}^\rho$ denotes a smoothing operator $\Hil_{s-h}\to\Hil_s$ for all $s\in\R$. Note that this indeed generalizes the class of time-periodic particle-field Hamiltonian systems in \cite{FabertLamoree21}, which model the interaction of a scalar wave field on the $d$-dimensional torus $T^d$ with a particle constrained to a submanifold $Q\subset T^d$: Here $M=T^*Q\subset T^*T^d=B$, $E=B\times\C$ and $\Hil=H^{\frac{1}{2}}(T^d,\C)$ can be viewed as a space of sections in the trivial bundle that are constant along leaves of the canonical Lagrangian foliation on $T^*T^d$ given by the cotangent fibres. Furthermore, the smoothing operator is given by convolution with a $C^h$-function $\rho$ which models the charge distribution of the particle. By contrast, recall that in this paper we consider the case where $(M,\omega_M)$ is closed. 

\begin{definition}
A map $F_t:\widetilde{M}\to\R$ is called $h$-\textup{regularizing} if it extends to a smooth map
\[
F_t:\widetilde{M}_{-h}\to\R,
\]
and it is called $\infty$-regularizing when it is $h$-regularizing for all $h\in\N$. 
\end{definition} 
With this we again define 
\begin{definition}
A nonlinearity $F_t:\widetilde{M}\to\R$ is called \textup{$A$-admissible} if it satisfies the following conditions:
\begin{enumerate}
\item $F_t$ is $T$-periodic with $(T,X)$ admissible. 
\item The nonlinearity is $h$-regularizing with $h> dr$. Here $r$ is the irrationality measure of $aT/2\pi$ and $d$ the order of the differential operator $A$. 
\item The extended map $F_t:\widetilde{M}_{-h}\to\R$ has bounded $C^\alpha$-norms for all $\alpha$.
\item $F_t$ has bounded support, in the sense that for every $k\in\N$ there exists $R_k>0$ such that $F_t(u)=0$ for all $u\in\widetilde{M}$ with $|(\pi_k\circ\pi_{\Hil})(u)|>R_k$.
\end{enumerate}
$F_t$ is called \emph{weakly $A$-admissible} when there exists $t$-dependent $c_t=c_{t+T}\in\Hil_h$ such that $u\mapsto F_t(u)-\langle c_t,\pi_{\Hil}(u)\rangle$ satisfies 1., 2., and 3. 
\end{definition}
Again we find 
\begin{proposition}
Let $\widetilde{F}_t:\widetilde{M}\to\R$ be a weakly $A$-admissible nonlinearity. Then 
\[
F_t(u):=\chi(|\pi_{\Hil}(u)|_{-h}^2)\widetilde{F}_t(u)
\]
with $h$ as in \Cref{admissible} condition 2, and where $\chi$ a smooth cut-off function with $\supp(\chi)\subseteq[0,R]$ for some $R>0$, is $A$-admissible.
\end{proposition} 
In this final chapter we want to show how our infinite-dimensional Gromov-Floer compactness result can be used to prove the existence of multiple different time-periodic solutions $u:\R\to\widetilde{M}$, $u(t+T)=u(t)$ of $\dot{u}=X_H(u)$ for the time-periodic infinite-dimensional Hamiltonian 
\begin{align*}
%\label{hamform}
H_t(u)=\frac{1}{2}\langle A\pi_{\Hil}(u),\pi_{\Hil}(u)\rangle+F_t(u)=:H_A(u)+F_t(u)
\end{align*}
given as the sum of some weakly $A$-admissible nonlinearity $F_t:M\times\Hil\to\R$ and the quadratic term $H_A$ defined by a linear, possibly unbounded, self-adjoint (differential) operator $A:\Hil\to\Hil$ which we again assume to be admissible in the sense of \Cref{admissibleA}. We want to emphasize that it is natural to assume that the unbounded free Hamiltonian $H_A$ is only depending on the $\Hil$-component of $u$, since the restriction of $H_A$ to every finite-dimensional subspace is a smooth Hamiltonian. %indeed a weakly $A$-admissible nonlinearity.
The flows of $H_t$ and of $H_A$ and $F_t$ are still related via
\begin{align*}
%\label{phiform}
\phi^{H_t}=\phi^{H_A+F_t}=\phi^{H_A\#G_t}=\phi^A\circ\phi^{G_t}
\end{align*}
where $G_t:=F_t\circ\phi^A_t$, and we will work with $\phi_t^A$ and $G_t$ rather than with $H_t=H_A+F_t$ because $H_A$ (and hence $H_t$) is only densely defined, whereas the flow $\phi_t^A$ is a symplectomorphism which is defined on the whole of $\Hil$ and hence on $\widetilde{M}$. Note that  $\phi^A_t\cdot u=(\pi_M(u),e^{itA}\cdot\pi_{\Hil}(u))$, that is, $\phi^A_t$ acts trivially on the first factor of $\widetilde{M}=M\times\Hil$.\\\par 
Note that in contrast to before, the infinite-dimensional phase space $\widetilde{M}=M\times\Hil$ inherits nontrivial topology from the finite-dimensional closed symplectic manifold $M$, which we will use to prove an infinite-dimensional version of the degenerate Arnold conjecture. Let 
\begin{align*}
  \textrm{cl}(M):=\max\{N+1: \exists\theta_1,\ldots,\theta_N\in \oplus_{d=1}^{\dim M} H^d(M)\backslash\{0\}\,\,\textrm{with}\,\, \theta_1\cup\ldots\cup\theta_N\neq 0\}   
\end{align*}
denote the cup-length of $M$ which is a topological invariant of $M$ and hence of $\widetilde{M}$. After fixing some collection $\theta_1,\ldots,\theta_N$ of $N=\textrm{cl}(M)-1$ non-zero cohomology classes of $M$ of non-zero degrees with $\theta_1\cup\ldots\cup\theta_N\neq 0$, we choose homology cycles $C_1,\ldots,C_N$ representing the chosen cohomology classes via Poincare duality, $\theta_1=\textrm{PD}[C_1],\ldots,\theta_N=\textrm{PD}[C_N]$. More precisely, we consider pseudo-cycles defined using Morse theory on $M$, see \cite{Schwarz2} for details.\\\par 
As we want to employ pseudoholomorphic curve methods, let $J_M$ denote an arbitrary $\omega_M$-compatible almost complex structure on $M$ and we denote by $J=J_M\times J_{\Hil}$ the product almost complex structure on $M\times\Hil$, where we again assume without loss of generality that the linear complex structure $J_{\Hil}$ on $\Hil$ is given by $i$. The following statement is a generalization of the main result in \cite{Schwarz2}, under the simplifying assumption that $\pi_2(M)=\{0\}$. 
\begin{theorem}
For every Hamiltonian $H_t(u)=H_A(u)+F_t(u)$ with $A$-admissible nonlinearity $F_t$ there exist $N$ $(\lfloor h/d \rfloor-1)$-times differentiable maps $\widetilde{u}=\widetilde{u}_1,\ldots,\widetilde{u}_N:\R\times\R\to M\times\Hil_{h-d(r-1)-1/2}\subset M\times\Hil$ ($h> dr$) satisfying 
the Floer equation and $\phi_T^A$-periodicity condition
\begin{align*}
\overline{\partial}_J\widetilde{u}+\nabla G_t(\widetilde{u})=0,\qquad\widetilde{u}(s,t+T)=\phi_{-T}^A\widetilde{u}(s,t).
\end{align*}
For every $\alpha=1,\ldots,N$ the Floer curve $\widetilde{u}_{\alpha}$ connects two different solutions $u=u^-_{\alpha},u^+_{\alpha}: \R\to M\times\Hil$ of 
\begin{align}
\label{ghameqn}
\dot{u}=X_t^G(u),\qquad u(t+T)=\phi_{-T}^A(u(t))
\end{align}
in the sense that there exist sequences $s_{\alpha,n}^\pm\in\R$ with $s_{\alpha,n}^\pm\to\pm\infty$ as $n\to\infty$ such that
\begin{align*}
%\label{asymptoticcondition}
\lim_{n\to\infty}\widetilde{u}_\alpha(s_{\alpha,n}^-,t)=u^-_\alpha(t),\qquad\lim_{n\to\infty}\widetilde{u}_\alpha(s_{\alpha,n}^+,t)=u^+_{\alpha}(t).
\end{align*}
Furthermore, since for the symplectic actions we have 
\begin{align*}
\mathcal{A}(u^-_1)<\mathcal{A}(u^+_1)\leq\mathcal{A}(u^-_2)<\ldots<\mathcal{A}(u^+_{N-1})\leq\mathcal{A}(u^-_N)<\mathcal{A}(u^+_N),    
\end{align*}
it follows that there are at least $N+1=\textrm{cl}(M)$ mutually different solutions of \eqref{ghameqn}.
\end{theorem}
Here the symplectic action $\mathcal{A}(u)$ of a solution $u:\R\to M\times\Hil$ of \eqref{ghameqn} is defined as $$\mathcal{A}(u)=\int_{D^2}\bar{u}^*\omega + \int_0^T G_t(u(t))\,dt,$$ where $\bar{u}$ is a filling of $u$, when viewed as a $T$-periodic orbit in the symplectic mapping torus $\R\times M\times\Hil/\{(t,u)\sim (t+T,\phi^A_{-T}(u))\}$; note that since $\pi_2(M)=\{0\}$, this definition is independent of the choice of $\bar{u}$. Following the proof of \Cref{mainthmsection1}, we get the following 
\begin{corollary}
For every Hamiltonian $H_t(u)=H_A(u)+F_t(u)$ with weakly $A$-admissible nonlinearity $F_t$ there exist $\textrm{cl}(M)$ mutually different $T$-periodic solutions $u=u_1,\ldots,u_{N+1}$ of regularity $h-d(r-1)-1/2$ for $h> dr$, %for every $s>s_1>s_0$, 
that is, $u:\R\to\widetilde{M}_{h-d(r-1)-1/2}\subset M\times\Hil$ with
\[
\partial_tu=JA\pi_{\Hil}(u)+J\nabla F_t(u),\qquad u(t+T)=u(t).
\]
\end{corollary}
For the proof we follow the strategy from before and use the existence of Floer curves in finite dimensions. More precisely, for every $k\in\N$ let $F^k_t:M\times\C^{2k+1}\to\R$ denote the restriction of $F_t:\widetilde{M}\to\R$ to the finite-dimensional submanifold $M\times\C^{2k+1}\subset M\times\Hil$. Note that $F^k_t$ now has bounded support in $M\times B_{R_k}(0)$ and we define again $G^k_t:=F^k_t\circ\phi_t^A$. Let $\M^k$ denote the moduli space of tuples $(\widetilde{u},\tau)$, where $\widetilde{u}:\R\times\R\to M\times\C^{2k+1}$ is again a Floer curve satisfying the asymptotic condition $\lim_{s\to\pm\infty}(\pi_{\Hil}\circ\widetilde{u})(s,t)=0$, the $\tau$-dependent Floer equation in $M\times\C^{2k+1}$ with periodicity condition
\begin{align*}
\overline{\partial}_J\widetilde{u}+\varphi_\tau(s)\nabla G_t(\widetilde{u})=0,\qquad\widetilde{u}(s,t+T)=\phi_{-T}^A\widetilde{u}(s,t)
\end{align*}
and the following \emph{intersection property}: every Floer curve $(\widetilde{u},\tau)$ in $\M^k$ is required to intersect all the cycles $C_1,\ldots,C_N$ in the sense that 
\begin{align*}
%\label{cycles}
(\pi_M\circ\widetilde{u})(2\tau\cdot\frac{1}{N+1},0)\in C_1,\,\, \ldots,\,\,(\pi_M\circ\widetilde{u})(2\tau\cdot \frac{N}{N+1},0)\in C_N.    \end{align*} 
\begin{lemma}
For every $\tau\in\N$ there is a Floer curve $(\widetilde{u},\tau)$ in $\M^k$. 
\end{lemma}
\begin{proof}
The proof is analogous to the proof in \Cref{fdsection}, so we will only focus on the differences and refer to \cite{Schwarz2} for further details. Assuming again transversality for the nonlinear Cauchy-Riemann operator for the moment, the moduli space of such pairs $(\widetilde{u},\tau)$ is a $1$-dimensional manifold. While in the proof of \Cref{fdresult} it was readily clear that there exists a Floer curve for $\tau=0$, here we have to additionally take the intersection property into account: since $\textrm{PD}[C_1]\cup\ldots\cup\textrm{PD}[C_N]\neq 0$, we may assume without loss of generality that $C_1,\ldots,C_N$ intersect transversally in a point, $C_1\cdot\ldots\cdot C_N=\{\textrm{point}\}$, so that the constant curve with image in $\{\textrm{point}\}\times\{0\}\subset M\times\C^{2k+1}$ is the unique Floer curve for $\tau=0$. Again Floer curves $(\widetilde{u},\tau)$ exist for all $\tau>0$ by Gromov-Floer compactness, as we can exclude bubbling-off of holomorphic spheres as well as breaking-off of cylinders for finite $\tau$. Note that, in order to exclude existence of holomorphic spheres we additionally use that $\pi_2(M)=\{0\}$. 
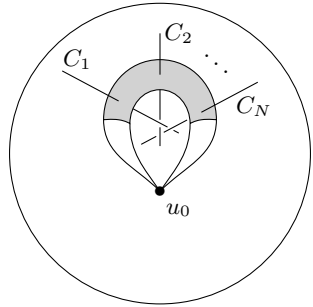
\begin{center}
\begin{tikzpicture}
\draw (0,.5) circle (2cm);
\draw (0,0) node {$\bullet$};
\draw (.25,-.25) node {$u_0$};
\draw (0,0) to[out=45,in=270] (.75,.95);
\draw (0,0) to[out=135,in=270] (-.75,.95);

\draw (0,0) to[out=55,in=270] (.4,.9);
\draw (0,0) to[out=125,in=270] (-.4,.9);

\

\fill[fill=gray!35!white] (.75,.95) to[out=90,in=0] (0,1.75) to[out=180,in=90] (-.75,.95) plot [smooth,tension=1] coordinates {(-.75,.95) (-.55,.95) (-.4,.9)} to[out=90,in=180] (0,1.35) to[out=0,in=90] (.4,.9) to[out=15,in=185] (.55,.95) to (.75,.95);
\draw (.75,.95) to[out=90,in=0] (0,1.75) to[out=180,in=90] (-.75,.95) plot [smooth,tension=1] coordinates {(-.75,.95) (-0.55,.95) (-.4,.9)} to[out=90,in=180] (0,1.35) to[out=0,in=90] (.4,.9) plot [smooth,tension=1] coordinates {(.4,.9) (.55,.95) (.75,.95)};

\draw (-1.3,1.6)--(-.55,1.2);
\draw (.25,.773)--(.18,.81);
\draw (.1,.85267)--(-.36,1.097);
\draw (0,2.1)--(0,1.55);
\draw (0,1.35)--(0,.95);
\draw (0,.85)--(0,.6);
\draw (1.3,1.45)--(.55,1.05);
%\draw (.36,.947)--(.18,.851994);
%\draw (.1,.80933)--(.05,.7827);
\draw (.36,.947)--(.05,.7827);
\draw (-.05,.7293)--(-.25,.623);

%\draw (.15,.835) node {$\bullet$};

\draw (-1.1,1.75) node {$C_1$};
\draw (.25,2.1) node {$C_2$};
\draw (.75,1.8) node {$\ddots$};
\draw (1.25,1.125) node {$C_N$};

%\draw [thick] (0,1.75) to[out=-70,in=70] (0,1.35);
%\draw (0,1.75) [thick, densely dashed] to[out=-110,in=110] (0,1.35);

%\draw (-1,.95) node {$\widetilde{u}$};
%\draw (0,1.95) node {$u_1$};
%
%\draw [thick] (0,0) to[out=135,in=270] (-.75,.95) to[out=90,in=180] (0,1.75);
%\draw [thick] (0,0) to[out=125,in=270] (-.4,.9) to[out=90,in=180] (0,1.35);
\end{tikzpicture}
\end{center}
Since we cannot expect transversality to hold, we again first need to approximate $J$ by a family of time-dependent almost-complex structures $J_t^\nu$ satisfying $(\phi_{-T}^A)_*J_t^\nu=J_{t+T}^\nu$, in the sense that $J_t^\nu\to J^0_t=i$ as $\nu\to\infty$. We emphasize that transversality now additionally includes that the evaluation map $\textrm{ev}=(\textrm{ev}_1,\ldots,\textrm{ev}_N)$ with $$\ev_\alpha:\M^k\to M,\,\, \widetilde{u}\mapsto(\pi_M\circ\widetilde{u})(2\tau\cdot\frac{\alpha}{N+1},0)\,\,\textrm{for}\,\,\alpha=1,\ldots,N$$ is transversal to $C_1\times\ldots\times C_N\subset M\times\ldots\times M$. \end{proof}

For every $k\in\N$ let again $\widetilde{u}^k:\R\times\R\to M\times\C^{2k+1}$ be a Floer curve in $\M^k$ for $\tau=k$, that is, $(\widetilde{u}^k,k)\in\M^k$. As before the idea is to apply our infinite-dimensional generalization of the Gromov-Floer compactness result to the sequence of Floer curves $\widetilde{u}^k$ in order to obtain a Floer curve in $\widetilde{M}=M\times\Hil$. More precisely, the proof of \Cref{finallemma} immediately leads to a proof of the following
\begin{lemma}
For every $\alpha=1,\ldots,N=\textrm{cl}(M)-1$, a subsequence of the sequence of \emph{shifted} Floer curves $$\widetilde{u}^k_{\alpha}:\R\times\R\to M\times\C^{2k+1}, \widetilde{u}^k_{\alpha}(s,t)= \widetilde{u}^k(s+2k\frac{\alpha}{N+1},t)$$ $\cloc$-converges (where $m=\floor{h/d}$) to a solution $\widetilde{u}=\widetilde{u}_{\alpha}:\R\times\R\to M\times \Hil$ of the Floer equation 
\[
\overline{\partial}_J\widetilde{u}+\nabla G_t(\widetilde{u})=0,\,\,\widetilde{u}(s,t+T)=\phi_{-T}^A\widetilde{u}(s,t)
\]
satisfying the intersection property $(\pi_M\circ\widetilde{u}_{\alpha})(0,t)\in C_{\alpha}$.
\end{lemma}
\begin{proof} The key observation is that, while in the unshifted case $\varphi_k(s,t)\to\varphi(s,t)$, in the shifted case we have $\varphi_k(s+2k\frac{\alpha}{N+1})\to 1$ for every $(s,t)\in\R\times\R$ as $k\to\infty$. We start by observing that we can write the finite-dimensional Floer curve as a tuple
\[
\widetilde{u}^k=(\widetilde{u}^{k,\ell},\widetilde{u}^{k,\ell}_\perp):\R\times\R\to (M\times\C^{2\ell+1})\times\C^{2k-2\ell}=M\times\C^{2k+1}\subset M\times\Hil,
\]
where $\widetilde{u}^{k,\ell}_\perp$ again denotes the normal component of $\widetilde{u}^k$. Again the extra statement needed for the proof is then that we still have for $m=\lfloor h/d \rfloor$ that
\begin{align*}
%\label{normalcomp1}
\sup_{k\geq\ell}\left\|\widetilde{u}^{k,\ell}_\perp\right\|_{C^{m-1}}\to0\quad\mathrm{as}\quad \ell\to\infty.
\end{align*}
Note that this relies on the fact that we have bounded derivatives, proven using bubbling-off, where we emphasize that the condition $\pi_2(M)=\{0\}$ ensures that the proof of \Cref{firstderiv} still goes through. Note that the latter also proves, using standard elliptic bootstrapping, that there is a  subsequence of $(\widetilde{u}^{k,\ell})_k$ of maps from $\R\times\R$ to $M\times \C^{2\ell+1}$ which $\cloc$-converges to a map $\widetilde{u}^\ell:\R\times\R\to\C^{2\ell+1}$ as $k\to\infty$ for all $\ell$. We stress that the maps $\widetilde{u}^{k,\ell}$ still take values in finite-dimensional compact manifold $M\times B^{2\ell+1}_{R_\ell}(0)$ by the bounded support condition and the maximum principle. Because we have locally bounded $W^{m+1,p}$-norms and hence, by elliptic bootstrapping and passing to a diagonal subsequence, local $W^{m,p}$-convergence, by Sobolev embedding we also have local convergence in the $C^{m-1}$-norm. Passing to a diagonal subsequence yet again, we obtain $\cloc$-convergence for all $\ell$ simultaneously, which, together with our result about the normal component proves that a subsequence of $\wt{u}^k:\R\times\R\to M\times\Hil$ is locally Cauchy. \end{proof}
But this implies that \Cref{lemmax} generalizes in the following sense 
\begin{lemma}
For every $\alpha=1,\ldots,N=\textrm{cl}(M)-1$ the limit Floer curve $\widetilde{u}_{\alpha}:\R\times \R\to\Hil$ satisfies the following asymptotic conditions: there exists sequences $s_{\alpha,n}^\pm\in\R$ with $s_{\alpha,n}^\pm\to\pm\infty$ as $n\to\infty$ such that
\begin{align*}
\lim_{n\to\infty}\widetilde{u}_{\alpha}(s_{\alpha,n}^-,t)=u^-_{\alpha}(t),\qquad\lim_{n\to\infty}\widetilde{u}(s_{\alpha,n}^+,t)=u^+_{\alpha}(t)
\end{align*}
in the $C^{m-1}$-sense ($m=\floor{h/d}$) where $u^-_{\alpha}$ and $u^+_{\alpha}$ are two different $\phi_T^A$-periodic orbits of $G_t$.
\end{lemma}
\begin{proof} The fact that $u^-_{\alpha}$ and $u^+_{\alpha}$ need to be different follows, as in \cite{Schwarz2}, from the fact that \begin{align*}
\mathcal{A}(u^+_{\alpha})-\mathcal{A}(u^-_{\alpha})=E(\widetilde{u}_{\alpha})=\int_{-\infty}^\infty\int_0^T\left|\partial_t\widetilde{u}_{\alpha}(s,t)-X_t^{G}(\widetilde{u}_{\alpha}(s,t))\right|^2dt\;ds
\end{align*}
with $E(\widetilde{u}_{\alpha})>0$ since $\widetilde{u}_{\alpha}$ must satisfy the intersection property $(\pi_M\circ\widetilde{u}_{\alpha})(0,t)\in C_{\alpha}$. \end{proof}

\appendix

\section{Sc-Hamiltonian flows}
\label{appendix}
Let us address the problem that the Hamiltonian $H_t$ is only densely defined, while the flow is defined on all of $\Hil$. In particular, we do not have a Hamiltonian flow in the usual sense. Rather, it is an \emph{sc-Hamiltonian flow}, which we define as follows.
\begin{definition}
A map $H:\Hil_h\to\R$ is called \textup{strongly sc\textsuperscript{1}} when the differential $dH:\Hil_h\times\Hil_h\to\R$ extends to a family of maps
\[
dH:\Hil_{h+\ell}\times\Hil_{h-\ell}\to\R
\]
for all $\ell\in\R$.
\end{definition}
Let $d\in\N$ be the order of the differential operator $A$, then note that $H_A$ is a map $H_A:\Hil_{d/2}\to\R$. % because $H_A(u)=\left|\sqrt{A}u\right|^2$ (in fact $H_A$ is smooth and quadratic as $H_A:W^{d/2,p}(\R/X\Z,\R)\to\R$). 
It is strongly sc\textsuperscript{1} because  $dH_A$ is given by
\[
dH_A(u)\cdot v=\langle Au,v\rangle
\]
and this defines a family of maps $dH_A:\Hil_{d/2+\ell}\times\Hil_{d/2-\ell}\to\R$ with $\ell\in\R$. %If $F_t:\Hil\to\R$ is smooth, then $H_t=H_A+F_t$ is still strongly sc\textsuperscript{1}. 
If we write $k=\ell-d/2$ then $dH_t:\Hil_{d+k}\times\Hil_{-k}\to\R$ for $k\in\R$. Note that $\omega$ induces and isomorphism $\omega:\Hil_k\stackrel{\sim}{\longrightarrow}\Hil_{-k}^*$ and so the (sc-) symplectic gradient $X^H_t$ defined by $\omega(X^H_t,\cdot)=dH_t$ is given by a family of maps $X^H_t:\Hil_{d+k}\to\Hil_k$ for all $k\in\R$. That is, $X^H_t$ is a scale morphism of order $d$ for all $k$. 
\begin{definition}
We say $\phi:\R\times\Hil\to\Hil$ is an \textup{sc-Hamiltonian flow of degree d} when
\begin{enumerate}
\item $\phi$ is sc$^\infty$ in the sense of \cite{hwzsc} for the Hilbert scale $(\Hil_{dn})_{n\in\N}$. In particular, the time-derivative defines a family of maps $\partial_t\phi:\Hil_{d(n+1)}\to\Hil_{dn}$ for all $n\in\N$.
\item There exists a strongly sc\textsuperscript{1} map $H_t:\Hil_{d/2}\to\R$ such that $\partial_t\phi=X_t^H$. 
\end{enumerate}
\end{definition}
The free flow $\phi_t^A$ is an sc-Hamiltonian flow. To show that we still get an sc-Hamiltonian flow after we have added the nonlinearity, it is sufficient to show that the flow of $F_t$ is smooth on $\Hil_k$ for all $k$. Then it is immediately sc-Hamiltonian. The first follows from the fact that $J\nabla F_t$ is smooth as a map from $\Hil_k$ to $\Hil_{k+h}$ for $h>0$ with uniform bounds, as the compact inclusion $\Hil_{k+h}\subset \Hil_k$ guarantees that the flow on $\Hil_k$ exists by Picard-Lindelöf. The nonlinearities in our examples satisfy this.\\\par

\bibliography{bib}{}
\bibliographystyle{alpha}
%\begin{thebibliography}%{10}
%%\bibliographystyle{vancouver} %Voor een bibliografie in de stijl van astronomy and astrophysics. Probeer ook alpha
%\bibliography{mscprojectbib2.bib} %Hier de naam van je bibtex file. BibTexTemplate is een voorbeeld!
% \end{thebibliography}

\end{document}